\documentclass{amsart}
\usepackage{graphicx}
\usepackage{amsfonts}
\newtheorem{tm}{Theorem}

\newtheorem{defi}{Definition}

\newtheorem{rem}{Remark}
\newtheorem{rems}{Remarks}
\newtheorem{conj}{Conjecture}
\newtheorem{lm}{Lemma}

\newtheorem{prop}{Proposition}
\newtheorem{nota}{Notation}

\begin{document}

\title{On the zero set of the partial theta function}
\author{Vladimir Petrov Kostov}
\address{Universit\'e C\^ote d’Azur, CNRS, LJAD, France} 
\email{vladimir.kostov@unice.fr}
\begin{abstract}
We consider the partial theta function 
$\theta (q,x):=\sum _{j=0}^{\infty}q^{j(j+1)/2}x^j$, where 
$q\in (-1,0)\cup (0,1)$ and either $x\in \mathbb{R}$ or $x\in \mathbb{C}$. 
We prove that for $x\in \mathbb{R}$, in each of the 
two cases $q\in (-1,0)$ and $q\in (0,1)$, its zero set consists of 
countably-many smooth curves in the $(q,x)$-plane each of which 
(with the exception of one curve for $q\in (-1,0)$) has a single point with 
a tangent line parallel to the $x$-axis. These points define double zeros 
of the function $\theta (q,.)$; their $x$-coordinates belong to the interval 
$[-38.83\ldots ,-e^{1.4}=4.05\ldots )$ for $q\in (0,1)$ and to the interval 
$(-13.29,23.65)$ for $q\in (-1,0)$. For $q\in (0,1)$, infinitely-many of the 
complex conjugate 
pairs of zeros to which the double zeros give rise cross the imaginary axis 
and then remain in the half-disk $\{ |x|<18$, Re\,$x>0\}$. For $q\in (-1,0)$, 
complex conjugate pairs do not cross the imaginary axis.


\end{abstract}
\maketitle 

\section{Introduction}

We consider the bivariate series $\theta (q,x):=\sum _{j=0}^{\infty}q^{j(j+1)/2}x^j$ 
which converges for $q\in (-1,1)$, $x\in \mathbb{C}$, and defines (for each 
fixed value of the parameter $q$) an entire function in $x$. We refer to 
$\theta$ as to a {\em partial theta function}. The terminology is justified 
by the fact that the series $\Theta (q,x):=\sum _{j=-\infty}^{\infty}q^{j^2}x^j$ 
defines the Jacobi theta function, and one has 
$\theta (q^2,x/q)=\sum _{j=0}^{\infty}q^{j^2}x^j$. The word ``partial'' hints at the 
fact that summation in $\theta$ is only partial (not from $-\infty$ to 
$\infty$, but only from $0$ to $\infty$). The function 
$\theta$ satisfies 
the differential equation

\begin{equation}\label{eqdiff}
2q\partial \theta /\partial q=x(\partial ^2/\partial x^2)(x\theta)=
x^2\partial ^2\theta /\partial x^2+2x\partial \theta /\partial x
\end{equation}
and the functional equation

\begin{equation}\label{eqfunct}
\theta (q,x)=1+qx\theta (q,qx)~.
\end{equation}

The interest in the function $\theta$ is explained 
by its applications in different 
areas. One of the most recent of them is about 
section-hyperbolic polynomials, i.e. 
real polynomials in one variable of degree $\geq 2$ having only 
real negative roots  
and such that when one deletes 
their highest-degree monomial, one obtains again a polynomial with all roots 
real negative. How $\theta$ arises in the context of such 
polynomials is explained in \cite{KoSh}. The explanation uses the notion 
of the {\em spectrum} of $\theta$ 
(see Section~\ref{secproperties}). The research on section-hyperbolic 
polynomials continued the activity in this domain which was marked by papers 
\cite{KaLoVi} and \cite{Ost}, and which were inspired by earlier results of 
Hardy, Petrovitch and Hutchinson 
(see \cite{Ha}, \cite{Pe} and \cite{Hu}). Section-hyperbolic polynomials 
are real, therefore the case when the parameter $q$ is real is of 
particular interest. The case $q\in \mathbb{C}$, $|q|<1$ (which 
has been studied by the author in \cite{KoPRSE1}, \cite{KoFAA} and 
\cite{KoDBAN2}) is not considered in the present paper.

The partial theta function is used in other domains as well. 
Such are asymptotic analysis (see \cite{BeKi}), statistical physics 
and combinatorics (see \cite{So}),   
Ramanujan-type $q$-series 
(see \cite{Wa}) and the theory 
of (mock) modular forms (see \cite{BrFoRh}); see also~\cite{AnBe}. 
Recently, new asymptotic results for Jacobi 
partial and false theta functions have been proved in~\cite{BFM}. They 
originate from Jacobi 
forms and find applications when considering the asymptotic expansions of 
regularized characters and 
quantum dimensions of the $(1,p)$-singlet algebra modules. The article 
\cite{CMW} is a closely related 
paper dealing with modularity, asymptotics and other properties 
of partial and false theta 
functions which are treated in the framework of  
conformal field theory and representation theory.
 
The present paper studies properties of the zero set of $\theta$. The case 
$q=0$ being trivial (with $\theta (0,x)\equiv 1$) one has to study in fact two 
different cases, namely $q\in (0,1)$ and $q\in (-1,0)$,  
in which the results are formulated in different ways. We present three 
different kinds of results. In Section~\ref{secproperties} 
we describe the set of 
real zeros of $\theta$ as a union of smooth curves in the $(q,x)$-space, 
see Theorem~\ref{tmgeom}. These results are further developed 
in Section~\ref{secfphikcor} by means of properties of certain 
functions in one variable; these properties are proved in 
Section~\ref{secfphik}. 

It is known that for each $q\in (-1,0)\cup (0,1)$ fixed, $\theta (q,.)$ has 
either only simple zeros or simple zeros and one double zero, 
see Theorems~\ref{tmknown1} and \ref{tmknown2}.    
In Section~\ref{secdouble} we prove that for $q\in (0,1)$, 
all double zeros of $\theta$ belong to 
the interval $[-38.83960007\ldots ,-4.055199967\ldots )$  
(Theorem~\ref{tmestim1}); for $q\in (-1,0)$, they belong to the interval 
$(-13.29,23.65)$ (Theorem~\ref{tmestim2}). 
In Section~\ref{seccomplex} we describe the behaviour 
of the complex conjugate pairs of $\theta (q,.)$. We show in 
Subsection~\ref{subseccomplexqneg} that in the case $q\in (-1,0)$, 
complex conjugate pairs do not cross the imaginary axis 
(Theorem~\ref{tmcplxneg}); hence each zero of $\theta$ 
remains in the left or right half-plane for all $q\in (-1,0)$. 
In Subsection~\ref{subseccomplexqpos} 
we show that as $q$ increases in $(0,1)$, infinitely-many 
complex conjugate pairs of 
$\theta$ go to the right half-plane, 
and after this remain in the half-disk $\{ |x|<18$, Re\,$x>0\}$.

\section{Geometry of the zero set of $\theta$
\protect\label{secproperties}}

First of all, we recall 
some known results in the case $q\in (0,1)$ (see \cite{KoBSM1}): 

\begin{tm}\label{tmknown1}
(1) For $q\in (0,\tilde{q}_1:=0.3092\ldots )$, all zeros of $\theta (q,.)$ are 
real, negative and distinct: $\cdots <\xi _2<\xi _1<0$. 

(2) There exist countably-many values 
$0<\tilde{q}_1<\tilde{q}_2<\cdots <1$ of $q$, 
where $\tilde{q}_j\rightarrow 1^-$ as $j\rightarrow \infty$, for which 
$\theta (q,.)$ has a multiple real zero $y_j$. For any $j\in \mathbb{N}$, 
this is the rightmost of the real zeros of $\theta$; 
it is a double zero of $\theta$. 
 
(3) For $q\in (\tilde{q}_j,\tilde{q}_{j+1})$ (we set $\tilde{q}_0:=0$), the 
function $\theta (q,.)$ has exactly $j$ complex conjugate pairs of zeros 
(counted with multiplicity). 
\end{tm}

\begin{defi}
{\rm We call {\em spectrum} of $\theta$ the set of values of $q$ 
for which $\theta (q,.)$ has at least one multiple zero. This notion is 
introduced by B.~Z.~Shapiro in~\cite{KoSh}.}
\end{defi}

\begin{rems}\label{remsdivers1}
{\rm (1) The zeros of $\theta$ depend continuously on $q$. Due to this, for 
$q\in (0,\tilde{q}_j)$, the order of its zeros 
$\cdots <\xi _{2j}<\xi _{2j-1}<0$ on the real line is well-defined. For no 
$q\in (0,1)$ does $\theta (q,.)$ have a nonnegative zero. 
For $q=\tilde{q}_j$, the zeros $\xi _{2j-1}$ and $\xi _{2j}$ 
coalesce and then become a complex conjugate pair for $q=(\tilde{q}_j)^+$; 
thus the indices $2j-1$ and $2j$ of the real zeros are meaningful exactly when 
$q\in (0,\tilde{q}_j]$. For $q\in (\tilde{q}_j,\tilde{q}_{j+1})$, one has 

$$\left\{ \begin{array}{lll}
\theta (q,x)>0&{\rm for}&
x\in (\xi _{2j+1},\infty )\cup 
(\cup _{k=j+1}^{\infty}(\xi _{2k+1},\xi _{2k}))\\ \\ 
\theta (q,x)<0&{\rm for}&x\in (\cup _{k=j+1}^{\infty}(\xi _{2k},\xi _{2k-1})~.
\end{array}\right.$$

(2) In the above setting, one has $-q^{-2j-2}<\xi _{2j+2}<\xi _{2j+1}<-q^{-2j-1}$, 
see Proposition~9 in~\cite{KoBSM1}.

(3) The function $\theta (\tilde{q}_j,.)$ has a local minimum 
at its double zero $y_j$. One has $\tilde{q}_j=1-\pi /2j+o(1/j)$ and 
$y_j=-e^{\pi}+o(1)$, where $e^{\pi}=23.14\ldots$, see \cite{KoRMC} 
or~\cite{KoBSM2}. Up to the sixth decimal, the first six spectral values 
$\tilde{q}_j$ 
equal $0.309249$, $0.516959$, $0.630628$, $0.701265$, $0.749269$, $0.783984$, 
see \cite{KoSh}.}
\end{rems} 

The analog of Theorem~\ref{tmknown1} in the case 
$q\in (-1,0)$ reads (see \cite{KoPRSE2}): 

\begin{tm}\label{tmknown2}
(1) For any $q\in (-1,0)$, the function $\theta (q,.)$ has infinitely-many 
negative and infinitely-many positive zeros. 

(2) There exists a sequence of values $\bar{q}_j$ of $q$ tending to $-1^+$ 
for which the function $\theta (\bar{q}_j,.)$ has a double real 
zero $\bar{y}_j$ (the rest of its real zeros being simple). For the rest of the 
values of $q\in (-1,0)$, $\theta (q,.)$ has no multiple real zeros. For 
$j$ large enough, one has $-1<\bar{q}_{j+1}<\bar{q}_j<0$. 

(3) For $j$ odd, one has $\bar{y}_j<0$, $\theta (\bar{q}_j,.)$ 
has a local minimum at $\bar{y}_j$ and $\bar{y}_j$ is the rightmost of the 
negative zeros of $\theta (\bar{q}_j,.)$. For $j$ even, 
one has $\bar{y}_j>0$, $\theta (\bar{q}_j,.)$ 
has a local maximum at $\bar{y}_j$ and $\bar{y}_j$ is the 
second from the left of the positive zeros of $\theta (\bar{q}_j,.)$.

(4) For $j$ sufficiently large and for $q\in (\bar{q}_{j+1},\bar{q}_j)$, 
the function $\theta (q,.)$ has exactly $j$ complex conjugate pairs of 
zeros counted with multiplicity.
\end{tm}

For $q\in (-1,0)$,  
the first six spectral values $\bar{q}_j$ equal (up to the sixth decimal) 
$-0.727133$, $-0.783742$, 
$-0.841601$, $-0.861257$, $-0.887952$ and $-0.897904$, see~\cite{KoPRSE2}.

\begin{rem}
{\rm For $q\in (-1,0)$ sufficiently close to $0$, all zeros of $\theta (q,.)$ 
are real. We denote them by $\cdots <\zeta _2<\zeta _1<0$ and 
$0<\eta _1<\eta _2<\cdots$. For $j=2\nu -1$ (resp. for $j=2\nu$), 
$\nu \in \mathbb{N}$, the zeros $\zeta _{2\nu -1}$ and $\zeta _{2\nu}$ 
(resp. $\eta _{2\nu}$ and $\eta _{2\nu +1}$) coalesce 
at $\bar{y}_{2\nu -1}$ when $q=\bar{q}_{2\nu -1}$ (resp. at $\bar{y}_{2\nu}$ 
when $q=\bar{q}_{2\nu}$). 
Thus the zero $\eta _1$ remains 
real positive and simple for all $q\in (-1,0)$.
This is deduced in \cite{KoPRSE2}, from the order of the 
quantities $\zeta _j$, $q\zeta _j$, $\eta _k$ and $q\eta _k$ 
on the real line (see Fig.~3 in~\cite{KoPRSE2}; the notation used 
in~\cite{KoPRSE2} is not the one we use here):}

$$\begin{array}{ccccccccccccccccc} 
\cdots &<&\zeta _4&<&\zeta _3&<&q\eta _4&<&q\eta _3&<&\zeta _2&<&\zeta _1&<
&q\eta _2&<&0\\ \\ 
0&<&\eta _1&<&q\zeta _1&<&q\zeta _2&<&\eta _2&<&\eta _3&
<&q\zeta _3&<&q\zeta _4&<&\cdots~~.
\end{array}$$
\end{rem}

Our first result is formulated as follows: 

\begin{tm}\label{tmgeom}
(1) Suppose that $q\in (0,1)$. For $j=1$, $2$, $\ldots$, consider 
the zeros $\xi _{2j-1}$ and $\xi _{2j}$ as functions in 
$q\in (0,\tilde{q}_j]$. Their 
two graphs together (in the $(q,x)$-plane) form a smooth curve $\Gamma _j$ 
having 
two parabolic branches $B_{2j-1}$ and $B_{2j}$ which are 
asymptotically equivalent to $x=-q^{-2j+1}$ and 
$x=-q^{-2j}$ as $q\rightarrow 0^+$. The curve $\Gamma _j$ 
has a single point $X_j$, namely for 
$q=\tilde{q}_j$, at which 
the tangent line is parallel to the $x$-axis.

(2) Suppose that $q\in (-1,0)$. For $\nu =1$, $2$, $\ldots$, consider 
the zeros $\zeta _{2\nu -1}$ and $\zeta _{2\nu}$ (resp. $\eta _{2\nu}$ 
and $\eta _{2\nu +1}$) as functions in 
$q\in [\bar{q}_{2\nu -1},0)$ (resp. $q\in [\bar{q}_{2\nu},0)$). Their 
two graphs together (in the $(q,x)$-plane) form a smooth curve $\Gamma _{\nu}^-$ 
(resp. $\Gamma _{\nu}^+$) having 
two parabolic branches $B_{2\nu -1}^-$ and $B_{2\nu}^-$ (resp. 
$B_{2\nu}^+$ and $B_{2\nu +1}^+$) which are 
asymptotically equivalent to $x=-q^{-4\nu +2}$ and 
$x=-q^{-4\nu}$ (resp. $x=-q^{-4\nu +1}$ and 
$x=-q^{-4\nu -1}$) as $q\rightarrow 0^-$. The curve $\Gamma _{\nu}^-$ 
(resp. $\Gamma _{\nu}^+$) has a single point $X_{\nu}^-$ (resp. $X_{\nu}^+$) 
such that for $q=\bar{q}_{2\nu -1}$ (resp. for $q=\bar{q}_{2\nu}$), the 
tangent line to $\Gamma _{\nu}^-$ at $X_{\nu}^-$ 
(resp. to $\Gamma _{\nu}^+$ at $X_{\nu}^+$) is parallel to the $x$-axis. 
The graph of the zero $\eta _1$ is asymptotically equivalent to $-q^{-1}$ as 
$q\rightarrow 0^-$ and one has $\eta _1\rightarrow 1^+$ as 
$q\rightarrow -1^+$.  
\end{tm}

\begin{rems}
{\rm (1) It is clear that the function $\xi _{2j-1}$ cannot be everywhere 
increasing 
on $(0,\tilde{q}_j]$ -- for $q$ close to $\tilde{q}_j$, the slope of the 
tangent line to its graph is positive whereas for $q$ close to $0$,  
it is negative. The graphs of the zeros $\xi _{2j-1}$ and $\xi _{2j}$ which 
coalesce for $q=\tilde{q}_j$ can be compared with the graphs of $\pm \sqrt{q}$ 
at $0$. Similar remarks can be made about the zeros $\zeta _j$ and $\eta _j$.

(2) The curves $x=-q^{-s}$ can be considered as curvilinear asymptotes to the 
zero set of $\theta$.}
\end{rems}

\begin{conj}
The curve $\Gamma _j$ from Theorem~\ref{tmgeom} has a single point $Q_j$ 
at which 
the tangent line is parallel to the $q$-axis, a single 
inflection point $I_j$ and a single point $D_j\in \Gamma _j$ at which 
one has $\theta (q,-q^{-2s+1/2})=0$.  
The order of the points and branches of $\Gamma _j$ 
is the following one: $B_{2j-1}$, $Q_j$, $D_j$, $X_j$, $I_j$, $B_{2j}$. 
The function $\xi _{2j}$ 
is everywhere increasing on $(0,\tilde{q}_j]$.
\end{conj}

On Fig.~\ref{zerosetPTF} we show parts of the curves $\Gamma _1$, $\Gamma _2$ 
and $\Gamma _3$ (drawn in solid line) and of the graphs of the functions 
$x=-q^a$ for $a=0.5$, $-1.5$ (drawn in solid), $-2.5$, $-3,5$ 
(drawn in dashed), $-1$ and $-2$ (drawn in dotted line). On 
Fig.~\ref{zerosetPTF} we show also the horizontal dash-dotted 
lines $q=0.26$ and $q=0.4$. 
We say that the part of the curve $x=-q^{-1.5}$ corresponding to $q\in (0,0.26]$ 
is {\em inside} and the part corresponding to $q\in [0.4,1)$ 
is {\em outside} the curve $\Gamma _1$. 

 \begin{figure}[htbp]
\centerline{\hbox{\includegraphics[scale=0.7]{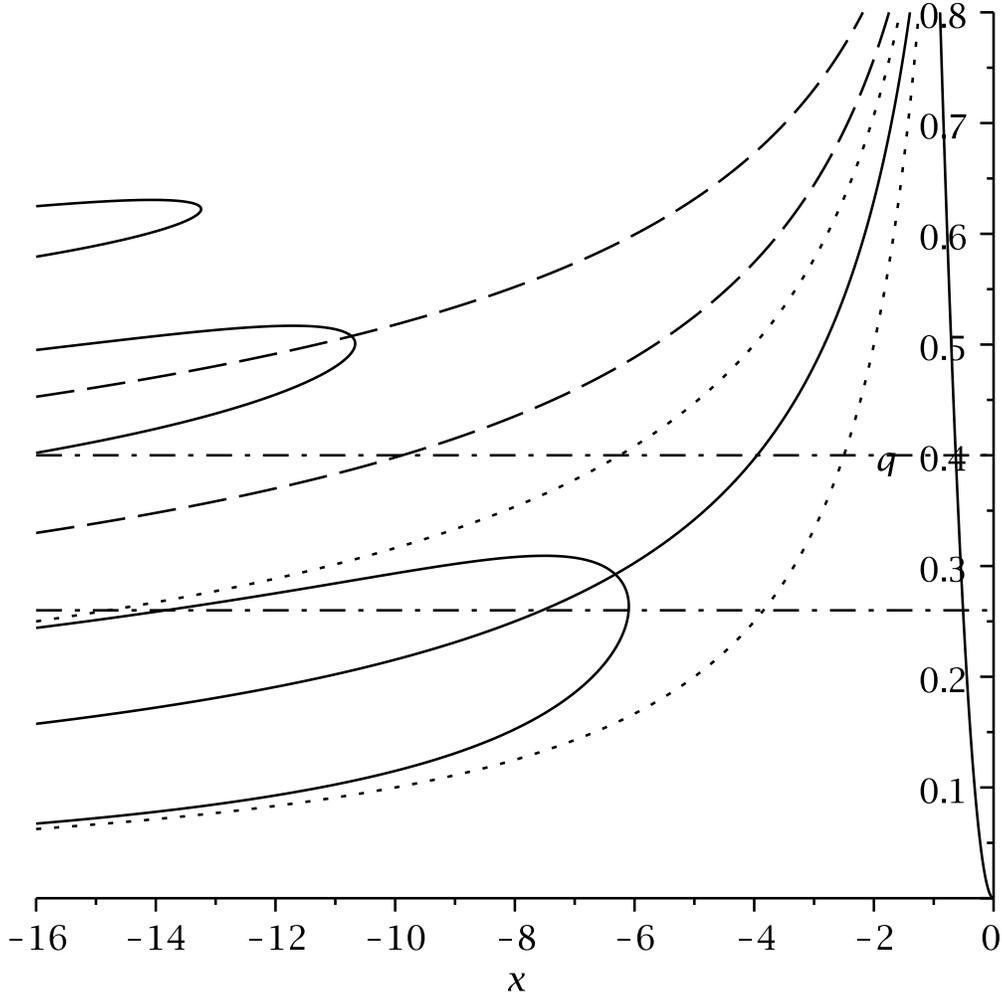}}}
    \caption{The curves $\Gamma _1$, $\Gamma _2$ 
and $\Gamma _3$ and the graphs $x=-q^a$ for $a=0.5$, $-1.5$, $-2.5$, $-3,5$, 
$-1$ and $-2$.}
\label{zerosetPTF}
\end{figure}
 
On Fig.~\ref{zerosetqneg} we show for $q\in (-1,0)$ the real-zero set 
of $\theta$ 
(in solid line) and the curves $x=-q^{-a}$, $a=1$, $\ldots$, $8$ (in dashed line 
for $a=1$, $2$, $5$ and $6$, and in dotted line for $a=3$, $4$, $7$ and $8$).  

 \begin{figure}[htbp]
\centerline{\hbox{\includegraphics[scale=0.7]{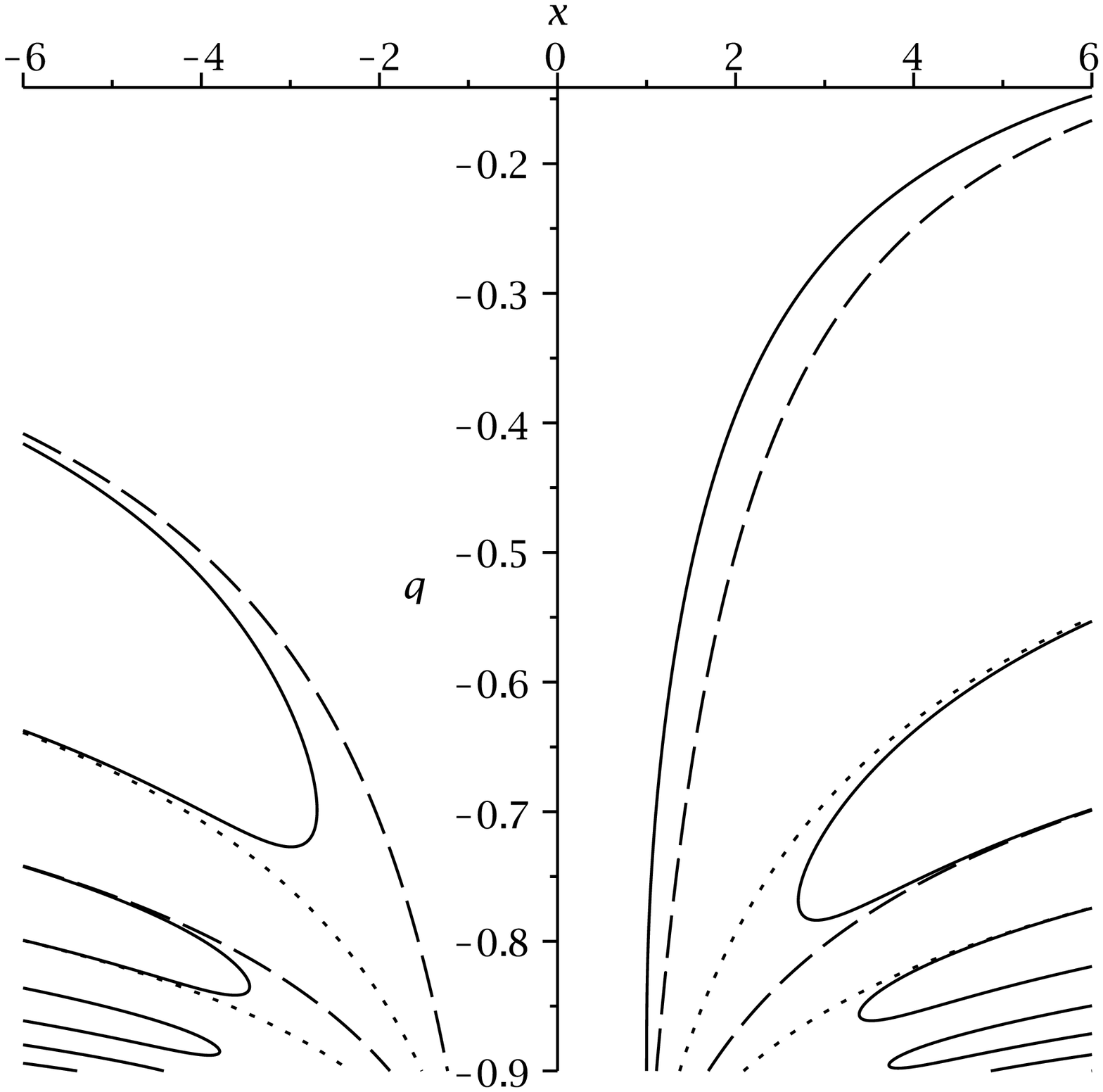}}}
    \caption{The zero set of $\theta$ for $q\in (-1,0)$ and the 
curves $x=-q^{-a}$ for $a=1$, $\ldots$, $8$.}
\label{zerosetqneg}
\end{figure}

\begin{rems}\label{remsdivers2}
{\rm (1) Suppose that $q\in (0,1)$. 
Inside (resp. outside) each curve $\Gamma _j$ one has 
$\theta (q,x)<0$ (resp. $\theta (q,x)>0$).

(2) One can check numerically that $\Gamma _1\subset \{ x\leq -6.095\}$. One 
can conjecture that any real zero of $\theta (q,.)$, 
for any $q\in (0,1)$, is smaller than $-6.095$. 

(3) Suppose that $q\in (-1,0)$. Inside each curve 
$\Gamma _{\nu}^-$ (resp. $\Gamma _{\nu}^+$) one has 
$\theta (q,x)<0$ (resp. $\theta (q,x)>0$). 
For $x<\eta _1$ (resp. $x>\eta _1$) and $(q,x)$ outside the curves 
$\Gamma _{\nu}^-$ (resp. $\Gamma _{\nu}^+$) one has 
$\theta (q,x)>0$ (resp. $\theta (q,x)<0$).

(4) One can check numerically that $\Gamma _1^-\subset \{ x\leq -2.699\}$. 
One can conjecture that any negative real zero of $\theta (q,.)$, for 
any $q\in (-1,0)$, is smaller than $-2.699$.

(5) For any $q\in (-1,0)$, the function $\theta (q,.)$ has no real zero 
in the interval $[-1,1]$. Indeed, one has $\theta (q,0)=1\neq 0$.  
For $x\in (0,\eta _1)$, one has $\theta (q,x)>0$, see part (3) 
of these remarks. For $x\in [-1,0)$, one 
obtains $\theta (q,x)=1+qx\theta (q,qx)$, where $qx\in (0,1)$ hence 
$qx\theta (q,qx)>0$ and $\theta (q,x)>0$.} 
\end{rems}


\begin{proof}[Proof of Theorem~\ref{tmgeom}:]
Part (1). The claims about the branches $B_{2j-1}$ and $B_{2j}$ follow from 
Theorem~4 in~\cite{KoBSM1}; for the branches $B_{\nu}^{\pm}$ this follows from  
part (1) of Theorem~1 in~\cite{KoDBAN1}. 
Smoothness of $\Gamma _j$ has to be proved only at $X_j$, everywhere else 
$\Gamma _j$ is the graph of a simple zero of $\theta$ which depends smoothly 
on $q$. For $q=\tilde{q}_j$, the function $\theta$ has a double zero 
at $\xi _{2j-1}=\xi _{2j}$, so 
$(\partial \theta /\partial x)(\tilde{q}_j,\xi _{2j-1})=0$ and 
$(\partial ^2\theta /\partial x^2)(\tilde{q}_j,\xi _{2j-1})\neq 0$. This implies 
$(\partial \theta /\partial q)(\tilde{q}_j,\xi _{2j-1})\neq 0$, 
see (\ref{eqdiff}), from which smoothness of $\Gamma _j$ at $X_j$ follows.
Simplicity of the zeros $\xi _{2j-1}$ and $\xi _{2j}$ for $q\in (0,\tilde{q}_j)$ 
excludes tangents parallel to the $x$-axis on $\Gamma _j\backslash \{ X_j\}$. 
 
Part (2). The claims about the curves $\Gamma _j^{\pm}$ 
are proved by analogy with 
the claims about the curves $\Gamma _j$. 
%
By Proposition~4.5 of \cite{KoPRSE2}, one has $\theta (q,-q^{-1})<0$. 
(On Fig.~\ref{zerosetqneg} this corresponds to the fact that the graph of 
$\eta _1$ is to the left of the dashed curve $x=-q^{-1}$.) 
We show that $\theta (q,1)>0$ from which follows that 
$\eta _1\rightarrow 1^+$ as $q\rightarrow -1^+$. 
For $|q|<1$, one has 

$$\theta (q,1)=\sum _{j=0}^{\infty}q^{j(j+1)/2}=
\frac{1-q^2}{1-q}\cdot \frac{1-q^4}{1-q^3}\cdot \frac{1-q^6}{1-q^5}\cdots ~,$$
see Problem 55 in Part I, Chapter 1 of \cite{PoSz}. For $q\in (-1,0)$, 
all factors 
in the right-hand side are positive, hence $\theta (q,1)>0$. 
\end{proof}

\section{The functions $\varphi _k$\protect\label{secfphik}}

In the present section we consider some functions in one variable which 
play an important role in the proofs in this paper:

\begin{equation}\label{defifphik}
\varphi _k(q):=\theta (q,-q^{k-1})=\sum _{j=0}^{\infty}(-1)^jq^{A_j}~,~~~ 
A_j:=kj+j(j-1)/2~,~~~k\in \mathbb{R}~.
\end{equation} 
In the notation for $A_j$ we skip the parameter $k$ 
in order not to have too many indices. We prove the following theorem:

\begin{tm}\label{tmfuncphik}
(1) For $k>1$, the function $\varphi _k$ is of the class 
$C^1_{(0,1)}\cap C_{[0,1]}$; its right 
derivative at $0$ exists and equals $0$. For $k>0$, its left derivative at $1$ 
exists and equals $-(2k-1)/8$.

(2) For $k>0$ sufficiently large, the function $\varphi _k$ is decreasing on 
$[0,1]$.
\end{tm}
We prove part (1) of the theorem after the formulations 
of Propositions~\ref{propconv1} and \ref{propconv2}, 
and part (2) at the end of the section. 

\begin{rems}\label{remsphik}
{\rm (1) The functions $\varphi _k$ satisfy the functional equation

\begin{equation}\label{eqphik}
\varphi _k=1-q\varphi _{k+1}~.
\end{equation}
One can deduce from this equation that the formula for the left derivative 
at $1$ remains valid for all $k\in \mathbb{R}$. For $k\in \mathbb{Z}$, 
the function $\varphi _k$ belongs to the class $C^1_{(0,1)}\cap C_{[0,1]}$ 
(the negative powers of $q$ cancel). For 
$k<0$, $k\not\in \mathbb{Z}$, one has $\varphi _k(q)\rightarrow \infty$ or 
$\varphi _k(q)\rightarrow -\infty$ as $q\rightarrow 0^+$ depending on the 
parity of $[k]$ (the integer part of $k$). 
 
(2) We remind that: 
\vspace{1mm}

{\em i)} For $k>0$ 
(resp. for $k>1$) and $q\in (0,1)$, one has $\varphi _k<1/(1+q^k)$ 
(resp. $1/(1+q^{k-1})<\varphi _k$), and that 
$\lim _{q\rightarrow 1^-}\varphi _k(q)=1/2$, see~\cite{KoBSM1}. 
\vspace{1mm}

{\em ii)} When $\varphi _k(q)$ is considered as a function of $(q,k)$, then 
for $k>0$ and $q\in (0,1)$, one has 
$\partial \varphi _k/\partial k>0$, see~\cite{KoPRSE2}.}
\end{rems}

Consider the functions 

$$\Phi _m:=(\sum _{j=0}^{m-1}(-1)^jq^{A_j})+(-1)^mD_m~,~~~\, m\in \mathbb{N}~, 
~~~\, {\rm where}~~~\, D_m:=q^{A_m}/(1+q^{k+m-1/2})~.$$ 
Set $S_m:=(1+q^{k+m-1/2})(1+q^{k+m+1/2})$. 
Hence 

\begin{equation}\label{telescope}
\begin{array}{lcl}
\Phi _{m+1}&=&\Phi _m+\Psi _m~,\\ \\{\rm where}&&\\ \\    
\Psi _m&:=&(-1)^mq^{A_m}+(-1)^{m+1}D_{m+1}-(-1)^mD_m\\ \\ 
&=&(-1)^mq^{A_m}(q^{k+m-1/2}(1+q^{k+m+1/2})-q^{k+m}(1+q^{k+m-1/2}))/S_m\\ \\ 
&=&(-1)^mq^{A_m+k+m-1/2}(1-q^{1/2})(1-q^{k+m})/S_m~.
\end{array}
\end{equation} 

\begin{lm}\label{lmvalue}
For $m\in \mathbb{N}$, one has $\Phi _m'(1)=-(2k-1)/8$.
\end{lm}

\begin{proof}
Indeed, for $m=1$ this can be checked directly. For arbitrary $m\in \mathbb{N}$ 
this follows from $D_m'(1)=A_m/2-(k+m-1/2)/4=km/2+m^2/4-m/2-k/4+1/8$ hence 

$$\begin{array}{ccl}
\Psi _m'(1)&=&(-1)^{m+1}\left( k(m+1)/2+(m+1)^2/4-
(m+1)/2-k/4+1/8\right) \\ \\&&+(-1)^mA_m-(-1)^m(km/2+m^2/4-m/2-k/4+1/8)\\ \\ 
&=&(-1)^m\{ -\left( k(m+1)/2+(m+1)^2/4-(m+1)/2-k/4+1/8\right) \\ \\&&-
(km/2+m^2/4-m/2-k/4+1/8)+km+m(m-1)/2\} \\ \\ 
&=&0~.\end{array}$$
By induction on $m$ one concludes that 
$\Phi _m'(1)=-(2k-1)/8$ for $m\in \mathbb{N}$.
\end{proof}

$$\begin{array}{lccl}
{\rm Set}&T_m&:=&(1+q^{k+m-1/2})(1+q^{k+m+1/2})(1+q^{k+m+3/2})\\ \\ {\rm and}&  
U_m&:=&(-1)^mq^{A_m+k+m-1/2}/T_m~.\end{array}$$ 
We consider the sum $\Delta _m:=\Psi _m+\Psi _{m+1}$, because due to 
the opposite signs of its two terms, one obtains better estimations for the 
convergence of certain functional series:

\begin{equation}\label{DeltaPsi}
\begin{array}{ccl}
\Delta _m&=&(-1)^mq^{A_m+k+m-1/2}(1-q^{1/2})
((1+q^{k+m+3/2})(1-q^{k+m})\\ \\ &&-q^{k+m+1}(1+q^{k+m-1/2})(1-q^{k+m+1}))/T_m\\ \\ 
&=&U_m(1-q^{1/2})(1-q^{k+m})((1+q^{k+m+3/2})-q^{k+m+1}(1+q^{k+m-1/2}))\\ \\ 
&&-U_m(1-q^{1/2})q^{2k+2m+1}(1+q^{k+m-1/2})(1-q)\\ \\ 
&=&U_m(1-q^{1/2})(K_m+L_m+M_m)~,\end{array}\end{equation}
where 

$$\begin{array}{cl}
K_m:=(1-q^{k+m})(1-q^{k+m+1})~,~~~&
L_m:=q^{k+m+3/2}(1-q^{k+m})(1-q^{k+m-1})\\ \\  {\rm and}& 
M_m:=-q^{2k+2m+1}(1+q^{k+m-1/2})(1-q)~.\end{array}$$ 

\begin{prop}\label{propconv1}
The series $\Delta _1+\Delta _3+\Delta _5+\cdots$ and 
$\Delta _2+\Delta _4+\Delta _6+\cdots$ are uniformly convergent for 
$q\in [0,1]$.
\end{prop}

\begin{prop}\label{propconv2}
The series $\Delta _1'+\Delta _3'+\Delta _5'+\cdots$ and 
$\Delta _2'+\Delta _4'+\Delta _6'+\cdots$ are uniformly convergent for 
$q\in [0,1]$. 
\end{prop}


\begin{proof}[Proof of part (1) of Theorem~\ref{tmfuncphik}]
The first two claims follow from the convergence of the series $\varphi _k$ 
for $q\in (0,1)$. The third claim results from Propositions~\ref{propconv1} 
and \ref{propconv2} and from Lemma~\ref{lmvalue}. Indeed, on every interval 
$[\alpha ,\beta ]\subset (0,1)$, 
the sequence of functions $\Phi _m$ converges uniformly to $\varphi _k$ as 
$m\rightarrow \infty$; one has $\Phi _m(1)=1/2=\varphi _k(1)$. 
For $m$ even, resp. for $m$ odd, one has 

$$\begin{array}{cclcll}
\Phi _m&=&1+\Psi _1+\cdots +\Psi _m&=&
1+\Delta _1+\Delta _3+\cdots +\Delta _{m-1}&{\rm ,~~~resp.}\\ \\ 
\Phi _m&=&1-q^k+\Psi _2+\cdots +\Psi _m&=&
1-q^k+\Delta _2+\Delta _4+\cdots +\Delta _{m-1}&.\end{array}$$
The existence of the left derivative at $1$ follows 
from Proposition~\ref{propconv2}; its value is implied by Lemma~\ref{lmvalue}. 
\end{proof} 

To prove Propositions~\ref{propconv1} and \ref{propconv2} we introduce 
some notation:

\begin{nota}\label{notafF}  
{\rm We denote by $f_m$ 
and $F_m$ functions respectively of the form 

$$f_m:=q^{C_m}(1-q^{B_m})~~~\, {\rm and}~~~\, 
F_m:=q^{C_m}(1-q^g)~~~\, ,~~~\, q\in [0,1]~~~\, ,~~~\, m\in \mathbb{N}~,$$ 
where $C_m:=am^2+bm+c$, 
$B_m=gm+h$, $a>0$, $b\geq 0$, $c\geq 0$, $g>0$ and $h\geq 0$.}
\end{nota}

We use the following lemma whose 
proof is straightforward: 

\begin{lm}\label{lmfF} 
(1) The function $f_m$ is positive-valued on $(0,1)$, $f_m(0)=f_m(1)=0$, 
its maximum is attained for $q=\alpha _m:=(C_m/(C_m+B_m))^{1/B_m}=1-O(1/m^2)$ 
and equals 

\begin{equation}\label{eqfm}
f_m(\alpha _m)=\left( \frac{C_m}{C_m+B_m}\right) ^{C_m/B_m}
\frac{B_m}{C_m+B_m} \left\{ 
\begin{array}{cl}\leq&\frac{B_m}{C_m+B_m}=\frac{gm+h}{am^2+(b+g)m+c+h}\\ \\ 
=&(e^{-1}g/am)(1+o(1))
\end{array}\right.
\end{equation} 
For $m$ sufficiently large, one has $\alpha _m<\alpha _{m+1}<1$.

(2)  The function $F_m$ is positive-valued on $(0,1)$, 
$F_m(0)=F_m(1)=0$. For $m$ sufficiently large,  
its maximum is attained for $q=\beta _m:=(C_m/(C_m+g))^{1/g}$ 
and equals 

\begin{equation}\label{eqFm}
F_m(\beta _m)=\left( \frac{C_m}{C_m+g}\right) ^{C_m/g}
\frac{g}{C_m+g} \left\{ 
\begin{array}{cl}\leq&\frac{g}{C_m+g}=\frac{g}{am^2+bm+c+g}\\ \\ =&
(e^{-1}g/am^2)(1+o(1))\end{array}\right.\end{equation} 
For $m$ sufficiently large, one has $\beta _m<\beta _{m+1}<1$.
\end{lm}

\begin{proof}[Proof of Proposition~\ref{propconv1}]
We use the representation (\ref{DeltaPsi}) of the functions $\Delta _m$. 
The function $U_m(1-q^{1/2})K_m$ is of the form $F_mV_m$ 
(with $F_m:=q^{A_m+k+m-1/2}(1-q^{1/2})$ and $V_m:=(-1)^mK_m/T_m$, hence 
with $a=1/2$, $b=k+1/2$ and $c=k-1/2$), where 
the function $|V_m|$ is bounded on $[0,1]$ by some constant independent of $m$ 
(one has $|V_m|\leq 1$).
Similar statements holds true for the functions $U_m(1-q^{1/2})L_m$ 
and $U_m(1-q^{1/2})M_m$. Hence $|\Delta _m|=O(1/m^2)$ (see part (2) of 
Lemma~\ref{lmfF}) from which the proposition follows.
\end{proof}

\begin{proof}[Proof of Proposition~\ref{propconv2}]
For $q\in [0,1/2]$, the uniform convergence of the two series results from 
d'Alembert's criterium, so we assume that $q\in (1/2,1)$. 
We set $W_m:=A_m+m+1/2=m^2/2+(k+1/2)m+1/2$ and 
$\tilde{U}_m:=U_m(1-q^{1/2})K_m$.  
Hence  

$$\begin{array}{lll}
\tilde{U}_m=(-1)^mR^*_mR^{\dagger}_mR^{\flat}_mq^{k-1}/T_m~~~,&{\rm where}&
R^*_m:=q^{W_m/3}(1-q^{1/2})~~~,\\ \\  
R^{\dagger}_m:=q^{W_m/3}(1-q^{k+m})&{\rm and}&
R^{\flat}_m:=q^{W_m/3}(1-q^{k+m+1})~.\end{array}$$
We similarly represent the function 
$U^{\circ}_m:=U_m(1-q^{1/2})L_m$ in the form 

$$\begin{array}{lll}
U^{\circ}_m=(-1)^mP^*_mP^{\dagger}_mP^{\flat}_mq^{k-1}/T_m~~~,&{\rm where}&
P^*_m:=q^{W_m/3}(1-q^{1/2})=R^*_m~~~,\\ \\  
P^{\dagger}_m:=q^{W_m/3}(1-q^{k+m})=R^{\dagger}_m&{\rm and}&
P^{\flat}_m:=q^{W_m/3+k+m+3/2}(1-q^{k+m-1})\end{array}$$
and finally we set 
$U^{\sharp}_m:=U_m(1-q^{1/2})M_m$ and 

$$\begin{array}{lll}
U^{\sharp}_m=(-1)^{m+1}Q^*_mQ^{\dagger}_mQ^{\flat}_mq^{k-1}/T_m~~~,&{\rm where}&
Q^*_m:=q^{W_m/3}(1-q^{1/2})=R^*_m~~~,\\ \\  
Q^{\dagger}_m:=q^{W_m/3}(1+q^{k+m-1/2})&{\rm and}&
Q^{\flat}_m:=q^{W_m/3+2k+2m+1}(1-q)~.\end{array}$$

The proposition results from the following lemma:

\begin{lm}\label{lminequalities}
There exist constants $c_i>0$, $i=1$, $2$ and $3$, such that for 
$q\in (1/2,1)$, one has $|(\tilde{U}_m)'|\leq c_1q^{k-1}/m^2$, 
$|(U^{\circ}_m)'|\leq c_2q^{k-1}/m^2$ and $|(U^{\sharp}_m)'|\leq c_3q^{k-1}/m^2$. 
\end{lm}
\end{proof}

\begin{proof}[Proof of Lemma~\ref{lminequalities}]
We differentiate the functions 
$\tilde{U}_m$, $U^{\circ}_m$ and $U^{\sharp}_m$ as products of functions. 
To prove the existence of the constants $c_i$ we obtain estimations for the 
moduli of the factors $R_m^*$, $R_m^{\flat}$, $\ldots$ 
and for the moduli of their 
derivatives. Consider 
first the function $\tilde{U}_m$. The factor $R^*_m$ is a function of the 
form $F_m$ (see Notation~\ref{notafF}), so one can apply part (2) 
of Lemma~\ref{lmfF} to obtain the estimation

\begin{equation}\label{eqestimR1}
|R^*_m|\leq (1/2)/(W_m/3+1/2)<3/2W_m~.
\end{equation}
One has $(R^*_m)'=(W_m/3q)R^*_m-q^{W_m/3-1/2}/2$. From inequality (\ref{eqestimR1}) 
one concludes that for $q\in (1/2,1)$, 
 
\begin{equation}\label{eqestimR2}
|(R^*_m)'|\leq 1/2q+1/2<3/2~.
\end{equation}
For the factor $T_m$ one gets

\begin{equation}\label{eqestimR3}
|1/T_m|\leq 1~~~\, {\rm and}~~~\, 
|(1/T_m)'|\leq |T_m'|/|T_m|^2\leq |T_m'|\leq 12k+12m+6~.
\end{equation}
One can apply part (1) of Lemma~\ref{lmfF} to the factor $R_m^{\flat}$ which is 
of the form $f_m$:

\begin{equation}\label{eqestimR4}
|R_m^{\flat}|\leq (k+m+1)/(k+m+1+W_m/3)<3(k+m+1)/W_m\leq 6/m
\end{equation}
(the rightmost inequality is checked directly) 
and, as 

$$(R_m^{\flat})'=(W_m/3q)R_m^{\flat}-(k+m+1)q^{W_m/3+k+m}~,$$ 
one deduces the estimation (using $q>1/2$) 

\begin{equation}\label{eqestimR5}
|(R_m^{\flat})'|\leq 2(k+m+1)+(k+m+1)=3(k+m+1)~.
\end{equation}
By complete analogy one obtains the inequalities

\begin{equation}\label{eqestimR6}
|R_m^{\dagger}|\leq 3(k+m)/W_m\leq 6/m~~~\, \, {\rm and}~~~\, \, 
|(R_m^{\dagger})'|\leq 3(k+m)~.
\end{equation}
From inequalities (\ref{eqestimR1}) and (\ref{eqestimR3}) results that 

\begin{equation}\label{eqestimR7}
|R_m^*(1/T_m)'|\leq (3/2W_m)(12k+12m+6)\leq 36/m
\end{equation}
(the rightmost inequality is to be checked directly). Hence for the products 
resulting from the differentiation of $\tilde{U}_m$ one obtains the 
following inequalities (using $|(q^{k-1})'|=|(k-1)q^{k-1}/q|\leq |2(k-1)q^{k-1}|$):

\begin{equation}\label{eqestimR8}
\begin{array}{l}
|(-1)^mR^*_mR^{\dagger}_mR^{\flat}_m(1/T_m)'q^{k-1}|\leq 
(36/m)(6/m)^2q^{k-1}=
6^4q^{k-1}/m^3\leq 6^4q^{k-1}/m^2~,\\ \\ 
|(-1)^m(R^*_m)'R^{\dagger}_mR^{\flat}_m(1/T_m)q^{k-1}|
\leq (3/2)(6/m)^2q^{k-1}=54q^{k-1}/m^2~,\\ \\ 
|(-1)^mR^*_m(R^{\dagger}_m)'R^{\flat}_m(1/T_m)q^{k-1}|\leq 
(3/2W_m)3(k+m)6q^{k-1}/m\leq 54q^{k-1}/m^2~,\\ \\ 
|(-1)^mR^*_mR^{\dagger}_m(R^{\flat}_m)'(1/T_m)q^{k-1}|\leq 
(3/2W_m)(6/m)3(k+m+1)q^{k-1}\leq 54q^{k-1}/m^2~,\\ \\ 
|(-1)^mR^*_mR^{\dagger}_mR^{\flat}_m(1/T_m)(q^{k-1})'|\leq 
(3/2W_m)(6/m)^22|k-1|q^{k-1}\leq 18q^{k-1}/m^2~.
\end{array}
\end{equation}
Thus $|\tilde{U}_m'|\leq c_1q^{k-1}/m^2$, where $c_1:=6^4+3\times 54+18=1476$. 

For the product $U^{\circ}_m$ we similarly obtain the inequalities

\begin{equation}\label{eqestimP}
\begin{array}{l}
|P^{\flat}_m|\leq (k+m-1)/(k+m-1+k+m+3/2+W_m/3)\leq 6/m~~~~~~~~~~\, 
{\rm and}\\ \\ 
|(P^{\flat}_m)'|\leq (W_m/3+k+m+3/2)|P_m^{\flat}/q|+(k+m-1)q^{k+m-2+k+m+3/2+W_m/3}\\ \\ 
\leq 2+k+m-1=k+m+1~.\end{array}
\end{equation}
(the rest of the factors are present in $\tilde{U}_m$ as well). 
Thus one obtains by complete analogy the inequality 
$|(U^{\circ}_m)'|\leq c_1q^{k-1}$ (i.e. one can set $c_2:=c_1$). 

\begin{nota}
{\rm We set $\Xi :=W_m/3+k+m-3/2$ and $\Lambda :=W_m/3+2k+2m+1$.}
\end{nota}

When considering the term $U^{\sharp}_m$, one obtains the inequalities about 
$Q_m^{\dagger}$:

\begin{equation}\label{eqestimQ1}
|Q_m^{\dagger}|\leq 2~~~\, {\rm and}~~~\, |(Q_m^{\dagger})'|\leq 
(W_m/3q)|Q_m^{\dagger}|+(k+m-1/2)q^{\Xi}\leq W_m+\Xi +1~.
\end{equation}
and the ones concerning $Q_m^{\flat}$:

\begin{equation}\label{eqestimQ2}
|Q_m^{\flat}|\leq 1/\Lambda ~~~\, \, {\rm and}~~~\, \, 
|(Q_m^{\flat})'|\leq \Lambda |Q_m^{\flat}|+q^{\Lambda}\leq 2~.
\end{equation}
Therefore the analogs of inequalities (\ref{eqestimR8}) read: 

\begin{equation}\label{eqestimQ3}
\begin{array}{l}
|(-1)^mQ^*_mQ^{\dagger}_mQ^{\flat}_m(1/T_m)'q^{k-1}|\leq 
(36/m)(2/\Lambda )q^{k-1}\leq 36q^{k-1}/m^2~,\\ \\ 
|(-1)^m(Q^*_m)'Q^{\dagger}_mQ^{\flat}_m(1/T_m)q^{k-1}|
\leq (3/2)(2/\Lambda )q^{k-1}\leq 18q^{k-1}/m^2\\ \\ 
|(-1)^mQ^*_m(Q^{\dagger}_m)'Q^{\flat}_m(1/T_m)q^{k-1}|\leq 
(3/2W_m)(W_m+\Xi +1)(1/\Lambda )q^{k-1}\leq 6q^{k-1}/m^2\\ \\ 
|(-1)^mQ^*_mQ^{\dagger}_m(Q^{\flat}_m)'(1/T_m)q^{k-1}|\leq 
(3/2W_m)2^2q^{k-1}\leq 12q^{k-1}/m^2\\ \\ 
 |(-1)^mQ^*_mQ^{\dagger}_mQ^{\flat}_m(1/T_m)(q^{k-1})'|\leq 
(3/2W_m)(2/\Lambda )2|k-1|q^{k-1}\leq 36q^{k-1}/m^2~. 
\end{array}
\end{equation}
Thus one can set $c_3:=36+18+6+12+36=108$.

\end{proof}

\begin{lm}\label{lmphik1/2}
For $k\geq 1/2$, the function $\varphi _k$ is decreasing 
on $[0,1/2]$.
\end{lm}

\begin{proof}
One has 
$\varphi _k'/q^{k-1}=\sum _{j=1}^{\infty}(-1)^j(kj+j(j-1)/2)q^{k(j-1)+j(j-1)/2}$. Our 
aim is to show that $\varphi _k'/q^{k-1}<0$ for $q\in [0,1/2]$ from which 
the lemma follows. Denote by $g$ the series obtained from $\varphi _k'/q^{k-1}$ 
by deleting its first three terms, and by $h:=(4k+6)q^{3k+6}$ the first term 
of $g$. For $q\in (0,1/2]$, the series $g$ is a Leibniz one. 
Indeed, it is alternating and 
the modulus of the 
ratio of two consecutive terms equals 

$$B_{k,j}:=\frac{(k(j+1)+j(j+1)/2)q}{kj+j(j-1)/2}\leq 
\frac{k(j+1)+j(j+1)/2}{2kj+j(j-1)}<1~;$$
the last inequality results from the inequalities 
$k(j+1)<2kj$ and $j(j+1)/2<j(j-1)$ which hold true for $j\geq 4$. Besides, 
for each $k$ fixed, one has $\lim _{j\rightarrow \infty}B_{k,j}=q\leq 1/2$. Hence 
for $q\in [0,1/2]$, one has $0\leq g(q)\leq h(q)$. So it suffices to show 
that for $q\in [0,1/2]$, 

\begin{equation}\label{eqg0}
g_0:=-k+(2k+1)q^{k+1}-(3k+3)q^{2k+3}+(4k+6)q^{3k+6}<0~.
\end{equation}
For $q\in [0,1/2]$ and when $k\geq 1/2$ is fixed, 
the quantity 

$$\alpha _k(q):=1-(4k+6)q^{k+3}/(3k+3)$$ 
is minimal for $q=1/2$. The quantity $\alpha _k(1/2)=1-(4k+6)/2^{k+3}(3k+3)$  
is minimal for $k=1/2$ and $\alpha _{1/2}(1/2)=0.84\ldots >0.84$. 
This observation 
allows to majorize the sum of the last two summands of $g_0$ (see (\ref{eqg0})) 
by $-0.84\times (3k+3)q^{2k+3}$. Now our aim is to 
prove that 

$$g_1(q):=-k+(2k+1)q^{k+1}-0.84\times (3k+3)q^{2k+3}<0$$
for $q\in [0,1/2]$, $k\geq 1/2$. The only zeros of the function $g_1'$ are $0$ 
and 

$$(((2k+1)(k+1))/(0.84\times (3k+3)(2k+3)))^{1/(k+2)}~.$$ 
For $k=1/2$, the latter quantity equals 
$0.52\ldots >1/2$; this quantity increases with $k$. For $q$ close to $0$, 
the function $g_1$ is increasing. Hence it is increasing on $[0,1/2]$ (for any 
$k\geq 1/2$ fixed) and 

$$\max _{[0,1/2]}g_1(q)=g_1(1/2)=
-k+(2k+1)/2^{k+1}-0.84\times (3k+3)/2^{2k+3}=:g_*(k)~.$$
Suppose first that $k\geq 1$. Then 

$$-k+(2k+1)/2^{k+1}\leq -k+(2k+1)/4=-(k-1/2)/2\leq 0~,$$ 
so $g_*(k)<0$. For 
$k\in [1/2,1)$, one has

\begin{equation}\label{eqg*}
g_*'=-1+(2-(2k+1)\ln 2)/2^{k+1}+(1.68\times (3k+3)(\ln 2)-2.52)/2^{2k+3}~.
\end{equation}
As $1.68\times (3k+3)(\ln 2)-2.52\leq 1.68\times 6\times (\ln 2)-2.52=
4.46\ldots <4.47$, 
the last summand of $g_*'$ (see (\ref{eqg*})) is 
$<4.47/2^4=0.279375$. The second summand is maximal for $k=1/2$ 
in which case it equals 

$$(2-2\ln 2)/2^{3/2}=0.2169777094\ldots <0.22~.$$
Thus $g_*'\leq -1+0.22+0.279375<0$ and $g_*$ is maximal for $k=1/2$. 
One finds that $g_*(1/2)=-0.0291432189\ldots <0$ which proves the lemma.
\end{proof}


\begin{proof}[Proof of part (2) of Theorem~\ref{tmfuncphik}]
For $q\in [0,1/2]$, the statement results from Lemma~\ref{lmphik1/2}, so 
we assume that $q\in [1/2,1]$. We use the equality 

\begin{equation}\label{eqinfty}
\varphi _k=1-q^k+\Delta _2+\Delta _4+\Delta _6+\cdots ~.
\end{equation}
Hence $\varphi _k'=-kq^{k-1}+\Delta _2'+\Delta _4'+\Delta _6'+\cdots$.  
The functions $\Delta _{2\nu}'$ are sums of terms each of which 
can be majorized by $q^{k-1}c/\nu ^2$, where the constant $c>0$ can be 
chosen independent of $k$, see Lemma~\ref{lminequalities}. Thus 

$$\varphi _k'\leq -q^{k-1}(k-4c\sum _{\nu =1}^{\infty}1/\nu ^2)~.$$
The difference can be made positive by choosing $k$ sufficiently large. 
This proves the theorem. 
\end{proof}

\section{Further geometric properties of the zero set
\protect\label{secfphikcor}}

Denote by $K^{\dagger}\subset \mathbb{R}$ the set 
$\cup _{j=1}^{\infty}(2j-1,2j)$. 

\begin{prop}\label{propKdagger}
For each $a\in K^{\dagger}$ sufficiently large, 
there exists a unique point $(q_a,-q_a^{-a})$, 
$q_a\in (0,1)$, 
such that $\theta (q_a,-q_a^{-a})=0$. For $a>0$, $a\not\in K^{\dagger}$, there 
exists no such point.
\end{prop}

\begin{rems}\label{remsKdagger}
{\rm (1) The statements of the proposition are illustrated by 
Fig.~\ref{zerosetPTF} -- 
the curve 
$x=-q^{-1.5}$ (with $1.5\in K^{\dagger}$) intersects the curve $\Gamma _1$ while 
the curve $x=-q^{-2.5}$ (with $2.5\not\in K^{\dagger}$) does not intersect 
any of the curves $\Gamma _s$, $s\in \mathbb{N}$.

(2) We denote by $\kappa ^{\triangle}\in \mathbb{N}$ a constant such that for 
$a\geq \kappa ^{\triangle}$, $a\in K^{\dagger}$, the first statement 
of Proposition~\ref{propKdagger} holds true. Hence there exists 
$q^{\triangle}\in (0,1)$ such that the curves $\Gamma _i$, 
$i\leq \kappa ^{\triangle}$, belong to the 
set $\{ x\leq 0$, $q\in (0,q^{\triangle}]\}$. Observe that the curve 
$\Gamma _j$ intersects the curves $x=-q^{-a}$ with $a\in (2j-1,2j)$, therefore 
the property this intersection to be a point is guaranteed for 
$j\geq (\kappa ^{\triangle}+1)/2$.}
\end{rems}

\begin{proof}
We set $\Theta :=\sum _{j=-\infty}^{\infty}q^{j(j+1)/2}x^j$ 
and $G:=\sum _{j=-\infty}^{-1}q^{j(j+1)/2}x^j$, hence $\theta =\Theta -G$. 
From the Jacobi triple product one gets 

\begin{equation}\label{eqJac}
\Theta (q,x)=\prod _{j=1}^{\infty}(1-q^j)(1+xq^j)(1+q^{j-1}/x)~.
\end{equation}
Hence $\Theta ^0:=
\Theta (q,-q^{-a})=\prod _{j=1}^{\infty}(1-q^j)(1-q^{j-a})(1-q^{j+a-1})$. 
For each $a\in K^{\dagger}$ fixed, each factor $1-q^j$, each factor $1-q^{j-a}$ 
with $j>a$, and each factor 
$1-q^{j+a-1}$ is positive and decreasing; there is an odd number of factors 
$1-q^{j-a}$ with $j<a$, so $\Theta <0$. Set $j_0=[a]$ (the integer part of $a$). 
Thus one can represent $\Theta ^0$ in 
the form 

$$\Theta ^0=-q^{-s}\prod _{j=1}^{\infty}((1-q^j)(1-q^{j+a-1}))
\prod _{j=j_0+1}^{\infty}(1-q^{j-a})\prod _{j=1}^{j_0}(1-q^{a-j})~,$$
where $s=\sum _{j=1}^{j_0}(a-j)=j_0(2a-j_0-1)/2>0$, and conclude that 
the function $q^s\Theta ^0$ is a minus product of positive and decreasing in 
$q$ factors, therefore it increases from $-\infty$ to $0$  
as $q$ runs over the interval $(0,1)$. 
 
One has $-q^sG(q,-q^{-a})=q^s(1-\varphi _a)$, 
that is, for $a>0$ sufficiently large, $-q^sG(q,-q^{-a})$ 
is the product of two positive increasing in $q$ 
functions (see Theorem~\ref{tmfuncphik}), 
hence it is positive and increasing, from $0$ for $q=0$ to $1/2$ 
for $q=1$, as $q$ runs over $(0,1)$. This means that, for $a>0$ sufficiently 
large, the function 
$q^s\theta (q,-q^{-a})$ is increasing from $-\infty$ to $1/2$ as $q\in (0,1)$, 
so there exists a unique value of $q$ for which it vanishes. 

If $a\in \mathbb{N}$, then one of the factors of $\Theta ^0$ is $0$ and 
$\theta (q,-q^{-a})=-G(q,-q^{-a})=q^a\varphi _{a+1}$ which is positive on $(0,1)$. 

If $a>0$, $a\not\in K^{\dagger}\cup \mathbb{N}$, then the number of negative 
factors in $\Theta (q,-q^{-a})$ is even, so both $\Theta (q,-q^{-a})$ and $-G$ 
are positive on $(0,1)$.
\end{proof}

\begin{prop}\label{prop1/2}
For $s\in \mathbb{N}$, consider the values of the parameter $q\in (0,1)$ 
for which $\theta (q,-q^{-2s+1/2})=0$. 
Then for these values one has $(\partial \theta /\partial x)(q,-q^{-2s+1/2})>0$.
\end{prop}

The proposition implies that, if for some value of $q$ the quantity 
$-q^{-2s+1/2}$ is a zero of $\theta$ (i.e. $\theta (q,-q^{-2s+1/2})=0$), then this  
can hold true for a zero $\xi _{2j-1}$ and not for a zero $\xi _{2j}$ 
of $\theta$. It would be 
interesting to (dis)prove that at an intersection point of the curves 
$\Gamma _s$ with $\tilde{Q}_s~:~x=-q^{-2s+1/2}$ 
the slope of the tangent line to $\Gamma _s$ 
is as shown on Fig.~\ref{zerosetPTF}. 

\begin{proof}
Consider first the polynomial 

$$P(q,x):=q+2q^3x+3q^6x^2+\cdots +(4s-2)q^{(2s-1)(4s-1)}x^{4s-3}$$ 
which is a truncation of $\partial \theta /\partial x$. 
For $x=-q^{-2s+1/2}$, its monomials 

$$jq^{j(j+1)/2}x^{j-1}~~~{\rm and}~~~(4s-2-j)q^{(4s-2-j)(4s-1-j)/2}
x^{4s-3-j}~~,~~j=0~,~\ldots ~,~2s-2~,$$ 
equal respectively
$j(-1)^{j-1}q^E$ and $(4s-2-j)(-1)^{4s-3-j}q^E$, where 

$$E=(j^2-(4s-2)j+4s-1)/2~,$$ 
and their sum equals $(4s-2)(-1)^{j-1}q^E$. 
For $x=-q^{-2s+1/2}$, its monomial $(2s-1)q^{s(2s-1)}x^{2s-2}$ 
equals $(2s-1)q^{-2s^2+4s-1}$. Thus 

$$P(q,-q^{-2s+1/2})=(2s-1)q^{-2s^2+4s-1}(1-2q^{1/2}+2q^2-\cdots -2q^{(2s-1)^2/2})~.$$
Consider now the function $Q(q,x):=(\partial \theta /\partial x)(q,x)-P(q,x)$, 
i.e. 

$$Q(q,x):=\sum _{j=4s-1}^{\infty}jq^{j(j+1)/2}x^{j-1}=
\sum _{j=0}^{\infty}(j+4s-1)q^{(j+4s-1)(j+4s)/2}x^{j+4s-2}~.$$

$$\begin{array}{ccclccclc}
{\rm We~set}&M_j&:=&q^{(j+4s-1)(j+4s)/2}x^{j+4s-2}&{\rm and}&Q&:=&Q_1+Q_2&,
\\ \\  
{\rm where}&Q_1&:=&(4s-2)\sum _{j=0}^{\infty}M_j&{\rm and}&Q_2&:=&
\sum _{j=0}^{\infty}(j+1)M_j&.\end{array}$$ 
Hence
$$
Q_1(q,-q^{-2s+1/2})=(4s-2)q^{-2s^2+4s-1}
\sum _{j=2s}^{\infty}(-1)^jq^{j^2/2}$$
and  
$$P(q,-q^{-2s+1/2})+Q_1(q,-q^{-2s+1/2})=(2s-1)q^{-2s^2+4s-1}\phi (q)~,
$$
where 

\begin{equation}\label{eqdefiphi}
\phi (q):=1+2\sum _{j=1}^{\infty}(-1)^jq^{j^2/2}~.
\end{equation} 
One has $\phi (q)>0$, see 
(\ref{eqphi}), therefore $P(q,-q^{-2s+1/2})+Q_1(q,-q^{-2s+1/2})>0$. 
Recall that for $k>1$ one has $\varphi _k(q)>1/(1+q^{k-1})\geq (1/2)$ 
(see part (2) of Remarks~\ref{remsphik}), so 

$$2\varphi _k-1=1-2q^k+2q^{2k+1}-2q^{3k+3}+\cdots \geq 0~.$$ 
The function $Q_2(q,-q^{-2s+1/2})$ equals

$$\begin{array}{cl}
&
q^{4s-1}-2q^{6s-1/2}+3q^{8s+1}-4q^{10s+7/2}+5q^{12s+7}-6q^{14s+23/2}+\cdots \\ \\ 
=&q^{4s-1}(2\varphi _{2s+1/2}-1)+q^{8s+1}(2\varphi _{2s+5/2}-1)+
q^{12s+7}(2\varphi _{2s+9/2}-1)+\cdots ~,\end{array}$$
so it is the sum of the nonnegative-valued functions 
$q^{4sj+2j^2-4j+1}(2\varphi _{2s+(4j-3)/2}-1)$ and  

$$(\partial \theta /\partial x)(q,-q^{-2s+1/2})=(P+Q_1+Q_2)(q,-q^{-2s+1/2})
>0~.$$
\end{proof}

\section{Bounds for the double real zeros of $\theta$
\protect\label{secdouble}}

\subsection{The case $q\in (0,1)$}

We remind that Theorem~\ref{tmknown1} introduces the double zeros $y_s$ 
of $\theta$. In this subsection we prove the following theorem:

\begin{tm}\label{tmestim1}
For $q\in (0,1)$ and for $s\geq 15$, 
all double real zeros $y_s$ 
of $\theta (q,.)$ belong to 
the interval $[-38.83960007\ldots ,-e^{1.4}=-4.055199967\ldots )$.
\end{tm}

We remind that all real zeros are negative, 
see part (1) of Remarks~\ref{remsdivers1}, 
and that it is likely an upper bound $\leq -6.095$ 
for all real zeros of $\theta$ to exist, 
see part (2) of Remarks~\ref{remsdivers2}. The lower bound $-38.83960007\ldots$ 
from the theorem 
cannot be made better than $-e^{\pi}=-23.14\ldots$, 
see part (2) of Remarks~\ref{remsdivers1}. In Lemmas~\ref{lmest2} 
and \ref{lmtechnical} (used in the proof of Theorem~\ref{tmestim1}) 
the results are formulated for $s\geq 3$. In the theorem we prefer $s\geq 15$, 
because this gives an estimation much closer to~$-e^{\pi}$.

\begin{proof}[Proof of Theorem~\ref{tmestim1}]
We justify the lower bound $-38.8\ldots$ first. 
We consider the curve $\tilde{Q}_s~:~x=-q^{-2s+1/2}$, $s\in \mathbb{N}$, 
see Fig.~\ref{zerosetPTF}.  
We find a value $0<q^{\flat}_s<1$ of $q$ 
such that the part of the curve $\tilde{Q}_s$ corresponding to 
$q\in (0,q^{\flat}_s]$ is inside the curve $\Gamma _s$. We remind that 
this is illustrated on 
Fig.~\ref{zerosetPTF}: the part of the curve $\tilde{Q}_1$ which 
corresponds to $q\in (0,0.26]$ lies inside and the part corresponding to 
$q\in [0.4,1)$ is outside the curve $\Gamma _1$ 
(the concrete numerical value $0.26$ is not $q^{\flat}_1$; it is chosen 
just for convenience). 

Consider the intersection points $V_s$ and $W_s$ of the line $q=q^{\flat}_s$ 
with the curves $\tilde{Q}_s$ and $\tilde{R}_s~:~x=-q^{-2s}$. On 
Fig.~\ref{zerosetPTF} an idea about the points $V_1$ and $W_1$ is given by  
the intersection points of the 
line $q=0.26$ with the curves $\tilde{Q}_1$ and 
$\tilde{R}_1$ 
(the latter is the higher of the two curves drawn in 
dotted line). Hence  
the point $W_s$ is more to the left than the point $X_s$ 
which is defined in part (1) of Theorem~\ref{tmgeom}. Indeed,  
consider the tangent line $\mathcal{T}_s$ 
to the curve $\Gamma _s$ at the point $X_s$ and the horizontal line 
$\mathcal{H}_s$ passing 
through a point $\Psi _s\in \tilde{Q}_s\cap \Gamma _s$. (For $s$ sufficiently 
large, the intersection $\tilde{Q}_s\cap \Gamma _s$ consists of exactly one 
point, see Proposition~\ref{propKdagger}. We do not claim that this 
is the case for all $s$, but our reasoning is applicable to any 
of the points $\Psi _s\in \tilde{Q}_s\cap \Gamma _s$.)  
The lines $\mathcal{T}_s$ and $\mathcal{H}_s$ intersect the curve 
$\tilde{R}_s$ at points $X_s^*$ and $\Psi _s^*$. For the $q$- and 
$x$-coordinates of these points we have the inequalities 

$$q(W_s)<q(\Psi _s)=q(\Psi _s^*)<q(X_s)=q(X_s^*)~~~\, {\rm and}~~~\, 
x(W_s)<x(\Psi _s^*)<x(X_s^*)<x(X_s)~.$$ 
Therefore finding a lower bound for 
the quantity $x(W_s)$ implies finding such a bound for 
$x(X_s)$ as well.    

Consider the function $\varphi _k$  (see (\ref{defifphik})) 
for $k$ of the form $-2s+3/2$, 
$s\in \mathbb{N}$. The quantities $A_j$ decrease 
for $j\leq 2s-1$, they  
increase for $j\geq 2s$ and $A_j=A_{4s-2-j}$, $j=0$, $\ldots$, $2s-2$. 
Recall that the function $\phi$ is defined by formula (\ref{eqdefiphi}). 
One checks directly that

\begin{equation}\label{eqs}
\varphi _{-2s+3/2}=q^{-(2s-1)^2/2}(-\phi (q)+q^{2s^2}\varphi _{(4s+1)/2}(q))~.
\end{equation}
We prove that for $q\in (0,q^{\flat}_s)$, one has $\phi >q^{2s^2}\varphi _{(4s+1)/2}$
or, equivalently,

$$L:=\ln \phi >B:= \ln (q^{2s^2}\varphi _{(4s+1)/2})~.$$ 
In \cite{PoSz}, Chapter 1, Problem 56, 
it is shown that 
\begin{equation}\label{eqphi}
\phi (q^2)=\prod_{k=1}^{\infty}((1-q^k)/(1+q^k))~.
\end{equation} 
Hence 

$$\phi (q)=\prod_{k=1}^{\infty}((1-q^{k/2})/(1+q^{k/2}))<
\prod_{r=1}^{\infty}(1-q^r)$$
(we ignore the factors $1-q^{k/2}<1$ for $k$ odd and all denominators 
$1+q^{k/2}>1$). 
We shall be looking for $q^{\flat}_s$ of the form 
$y:=1-\beta /(2s-1)$, $\beta >0$. 
Then 

\begin{equation}\label{eqbeta}
\begin{array}{ccccccc}
L&<&\sum _{r=1}^{\infty}\ln (1-y^r) 
&<&-\sum _{r=1}^{\infty}y^r-
(1/2)\sum _{r=1}^{\infty}y^{2r}&&\\ \\ 
&=&-y/(1-y)-y^2/2(1-y^2)&=&-y(2+3y)/2(1-y^2)&=:&L_0~.\end{array}
\end{equation}
We set $\eta :=2s-1$. Hence $y=1-\beta /\eta$ and  

$$L_0=-C(\eta -\beta)/2\beta ~~~\, \, ,~~~\, \, {\rm where}~~~\, \, 
C:=(5\eta -3\beta )/(2\eta -\beta )~.$$
On the other hand, by expanding $\ln y^D$, $D=2s^2=\eta ^2/2+\eta +1/2$,  
in powers of $1/\eta$ one gets  

\begin{equation}\label{eqbeta1}
\ln ((1-\beta /\eta )^D)=-(\beta /2)\eta -\beta -\beta ^2/4-K~,
\end{equation}
where $K=\sum _{j=1}^{\infty}(\beta ^j/2j+
\beta ^{j+1}/(j+1)+\beta ^{j+2}/2(j+2))/\eta ^j$. For $q\in [0,1]$, one has 
$\varphi _{(4s+1)/2}(q)\in [1/2,1]$ (see part (2) of Remarks~\ref{remsphik}), so 

\begin{equation}\label{eqeq}
B:=\ln ((y)^{2s^2}\varphi _{(4s+1)/2})=
-(\beta /2)\eta -\beta -\beta ^2/4-K-K_1~,
\end{equation}
where $K_1\in [0,\ln 2]$. With the above notation one has (see~(\ref{eqphi}))

$$\begin{array}{ccl}
L&=&\sum _{k=1}^{\infty}(\ln (1-q^{k/2})-\ln (1+q^{k/2}))\\ \\ 
&=&\sum _{k=1}^{\infty}((-q^{k/2}-(q^{k/2})^2/2-(q^{k/2})^3/3-\cdots )\\ \\ &&-
(q^{k/2}-(q^{k/2})^2/2+(q^{k/2})^3/3-\cdots ))\\ \\ 
&=&(-2)\sum _{k=1}^{\infty}(q^{k/2}+
(q^{k/2})^3/3+(q^{k/2})^5/5+\cdots ))\\ \\
&=&(-2)\sum _{k=1}^{\infty}q^{1/2}F_k(q)/(1-q^{1/2})(2k-1)~~~,\end{array}$$
where 

$$\begin{array}{cccclcc}
0&<&F_k(q)&=&q^{(k-1)/2}/(1+q^{1/2}+q+\cdots +q^{(k-1)/2})&&\\ \\ 
&&&=&1/(1+q^{-1/2}+q^{-1}+\cdots +q^{-(k-1)/2})&\leq&1/k~.\end{array}$$
Thus $L\geq (-2)(q^{1/2}(1+q^{1/2})/(1-q))S$, where 
$S:=\sum _{k=1}^{\infty}1/k(2k-1)=2\ln 2$. For $q=y=1-\beta /\eta$ 
(hence $1-q=\beta /\eta$, $q^{1/2}\leq 1$ and $1+q^{1/2}\leq 2$) one obtains 
the estimation 

$$L\geq (-8\ln 2)(\eta /\beta )~.$$
From this inequality we deduce the following lemma:

\begin{lm}\label{lmest2}
For $\beta =4\sqrt{\ln 2}=3.330218445\ldots$ and $s\geq 3$, one has $L>B$. 
Hence for $s\geq 3$, one can set $q^{\flat}_s:=1-4\sqrt{\ln 2}/(2s-1)$.
\end{lm}

We cannot allow the values $s=1$ and $s=2$, because in this case 
$y=1-\beta /\eta$ is negative. The $x$-coordinate of the point 
$W_s$ defined in the second paragraph of this proof equals 
$\lambda _s:=-(1-4\sqrt{\ln 2}/(2s-1))^{-2s}$.  

\begin{lm}\label{lmtechnical}
The functions $\Phi ^{\flat}:=-2x\ln (1-4\sqrt{\ln 2}/(2x-1))$ 
and $\exp (\Phi ^{\flat})$ are 
decreasing for $x\geq 3$. 
\end{lm}

Hence for $s\geq 15$, 
the lower bound of the sequence $\lambda _s$ equals 
$-\exp(\Phi ^{\flat}(15))=-38.83960007\ldots$.

\begin{proof}[Proof of Lemma~\ref{lmtechnical}]
One has 

$$(\Phi ^{\flat})'=
-2\ln (1-4\sqrt{\ln 2}/(2x-1))-16x\sqrt{\ln 2}/((2x-1)(2x-1-4\sqrt{\ln 2}))$$ 
hence $(\Phi ^{\flat})'\rightarrow 0$ as $x\rightarrow \infty$.  
Next, 

$$(\Phi ^{\flat})''=\frac{32((2+4\sqrt{\ln 2})x-1-4\sqrt{\ln 2})\sqrt{\ln 2}}
{(-2x+1+4\sqrt{\ln 2})^2(2x-1)^2}~,$$
which is positive for $x\geq 3$. As $(\Phi ^{\flat})'(3)=-2.5\ldots <0$, 
the function $(\Phi ^{\flat})'$ is negative on $[3,\infty )$. The same is true 
for $(\exp (\Phi ^{\flat}))'=(\exp (\Phi ^{\flat}))(\Phi ^{\flat})'$. 
\end{proof}

\begin{proof}[Proof of Lemma~\ref{lmest2}]
Indeed, set $L^*:=-L$ and $B^*:=-B$. 
It suffices to show that $L^*<B^*$ which results from 
$(8\ln 2)/\beta =\beta /2$ hence 

$$(8\ln 2)(\eta /\beta )<(\beta /2)\eta +\beta +\beta ^2/4$$
(we minorize $K$ and $K_1$ by $0$, see (\ref{eqbeta1})). 
\end{proof}

In order to justify the upper bound $-e^{1.4}$ we need the following lemma:

\begin{lm}\label{lmest1}
For $\beta \leq 1.4$ and $s\geq 14$, one has $L<L_0<B$.
\end{lm}

\begin{proof}
Indeed, consider the quantities $L_0^*=-L_0$ and $B^*=-B$. 
We show that 
$L_0^*>B^*$ from which the lemma follows. This is tantamount to 

\begin{equation}\label{ineq}
C(\eta -\beta)/2\beta >(\beta /2)\eta +\beta +\beta ^2/4+K+K_1~.
\end{equation}
We majorize $K_1$ by $\ln 2$. We observe that 
$C=2.5-(\beta /2)/(2\eta -\beta )$ 
is increasing in $\eta$ (i.e. in $s$) 
and decreasing in $\beta$. Therefore inequality (\ref{ineq}) results from 
the inequality 

\begin{equation}\label{ineq1}
C^{\dagger}(\eta -\beta)/2\beta >(\beta /2)\eta +\beta +\beta ^2/4+K+\ln 2~,
\end{equation}
where $C^{\dagger}=C|_{\beta =1.4,s=14}=2.486692015\ldots$. 
Inequality (\ref{ineq1}) can be given the equivalent form

$$(C^{\dagger}/\beta -\beta )(\eta /2)-C^{\dagger}/2>
\beta +\beta ^2/4+K+\ln 2~.$$
The coefficient $C^{\dagger}/\beta -\beta$ is positive 
and decreasing in $\beta$ while the right-hand side is increasing in $\beta$.  
The left-hand side is increasing in $s$ while the right-hand side is 
decreasing in it. Therefore it suffices to prove the last inequality 
(hence inequality (\ref{ineq1})) for 
$\beta =1.4$ and $s=14$. The left and right-hand sides of (\ref{ineq1}) 
equal respectively $3.835469849\ldots$ and $2.664996872\ldots$. 
The lemma is proved. 
\end{proof} 

To deduce from the lemma the upper bound from Theorem~\ref{tmestim1} we 
set $\mu _s:=-(1-1.4/(2s-1))^{-2s}$; we apply a reasoning similar to  
the one concerning the lower bound and the quantity $\lambda _s$. 
One has $\mu _{14}=-4.440852689\ldots$. 
The quantity $\mu _s$ increases with $s$ and 
$\lim _{s\rightarrow \infty}\mu _s=-e^{1.4}=-4.055199967\ldots$.

\end{proof}

\subsection{The case $q\in (-1,0)$}

We begin the present subsection with 
a result concerning the case $q\in (0,1)$. Recall that, for $q\in (0,1)$, 
the third spectral value equals 
$\tilde{q}_3=0.630628\ldots$. Hence $(\tilde{q}_3)^{-3}=3.98\ldots <4$.  

\begin{prop}\label{propfirsttwo}
For $q\in (\tilde{q}_3,1)$, the first two rightmost real zeros of 
$\theta (q,.)$ are $>-156$.
\end{prop}

\begin{proof}
Suppose that $q\in (\tilde{q}_k,\tilde{q}_{k+1}]$, $k\geq 3$. Then the two 
rightmost zeros of $\theta (q,.)$ are $\xi _{2k+2}$ and $\xi _{2k+1}$. They are 
defined for $q\in (0,\tilde{q}_{k+1}]$; for $q=\tilde{q}_{k+1}$ they coinside. 
For $q=\tilde{q}_k$, one has 

$$\begin{array}{ccccccccc}
-(\tilde{q})^{-2k-2}&<&\xi _{2k+2}&<&\xi _{2k+1}&<&-(\tilde{q}_k)^{-2k-1}&
<&-(\tilde{q}_k)^{-2k}\\ \\ &&&<&\xi _{2k}&=&
\xi _{2k-1}&<&-(\tilde{q}_k)^{-2k+1}~,\end{array}$$
see Fig.~\ref{zerosetPTF}. Observe that 
for $q\in [\tilde{q}_k,\tilde{q}_{k+1}]$, 
$k\geq 1$, the value of $-q^{-2k-2}$, 
the minoration of $\xi _{2k+2}$, is minimal when $q=\tilde{q}_k$.  
Hence 

$$(\tilde{q}_k)^{-3}\xi _{2k}<-(\tilde{q})^{-2k-2}<\xi _{2k+2}<\xi _{2k+1}~.$$ 
The factor 
$(\tilde{q}_k)^{-3}$ is maximal for $k=3$ whereas $-39<\xi _{2k}$ 
(see~Theorem~\ref{tmestim1}). This together with $(\tilde{q}_3)^{-3}<4$ implies 
$-156=4\times (-39)<\xi _{2k+2}<\xi _{2k+1}$. 
\end{proof}

The basic result of the present subsection is the following theorem:

\begin{tm}\label{tmestim2}
For $q\in (-1,0)$, all double zeros of $\theta (q,.)$ belong to the interval 
$(-13.29,23.65)$.
\end{tm}

\begin{proof}
It is explained in \cite{KoPRSE2} how for $q\in (-1,0)$ the simple real 
zeros of $\theta$ coalesce to form double ones and then complex conjugate 
pairs. We reproduce briefly the reasoning from \cite{KoPRSE2}. 

We set $v:=-q$ (hence $v\in (0,1)$) and 

\begin{equation}\label{eqpsi1psi2}
\begin{array}{ccccccccc}
\theta (q,x)&=&\theta (-v,x)&=&\psi _1+\psi _2&,&{\rm where}&&\\ \\ 
\psi _1(v,x)&:=&\theta (v^4,-x^2/v)&&{\rm and}&&\psi _2(v,x)&:=&
-vx\theta (v^4,-vx^2)~;\end{array}\end{equation}
the equality $\theta (-v,x)=\psi _1(v,x)+\psi _2(v,x)$ is to be 
checked directly. 
For $v$ fixed, the function $\psi _1$ is even while $\psi _2$ is odd. Denote 
by $y_{\pm k}$ and $z_{\pm k}$ the zeros of $\psi _1$ and $\psi _2$, where 

$$\begin{array}{cccccccccccccc}
y_k&=&-y_{-k}&,&y_{-k-1}&<&y_{-k}&<&0&<&y_k&<&y_{k+1}&,\\ \\ z_k&=&-z_{-k}&,&
z_{-k-1}&<&z_{-k}&<&0&<&z_k&<&z_{k+1}&,\\ \\ y_{\pm k}&=&vz_{\pm k}&.&&&&&&&&&&
\end{array}$$
For $v^4\in (0,\tilde{q}_1)$, all zeros of $\psi _1$ and all zeros of $\psi _2$ 
are simple (see part (1) of Theorem~\ref{tmknown1}). 
For small values of $v$, the zeros $y_{\pm k}$ and $z_{\pm k}$ are close to 
$\pm v^{-(4k-1)/2}$ and $\pm v^{-(4k+1)/2}$ respectively. 

Suppose first that $x<0$. The function $\psi _1$ (resp. $\psi _2$) 
is negative on the interval $(y_{-2\nu},y_{-2\nu +1})$ 
(resp. $(z_{-2\nu},z_{-2\nu +1})$) and positive on the interval 
$(y_{-2\nu -1},y_{-2\nu})$ (resp. $(z_{-2\nu -1},z_{-2\nu})$). For small values of 
$v$, the order of these points and of their approximations by powers of $v$ 
on the real line looks like this:

\begin{equation}\label{eqfirststring}
\begin{array}{cccccccc}
z_{-2\nu -1}&<&y_{-2\nu -1}&<&z_{-2\nu}&<
&y_{-2\nu}&<\\  
-v^{-4\nu -5/2}&&-v^{-4\nu -3/2}&+&-v^{-4\nu -1/2}&&-v^{-4\nu +1/2}&-\\ \\ 
z_{-2\nu +1}&<&y_{-2\nu +1}&<&0&.&&\\  
-v^{-4\nu +3/2}&&-v^{-4\nu +5/2}&+&&&&\end{array}
\end{equation}
The signs $+$ and $-$ in the second rows indicate intervals on which both 
functions $\psi _1$ and $\psi _2$ (hence $\theta (-v,.)$ as well) are positive 
or negative respectively. Thus for $v^4\in (0,\tilde{q}_1)$, $\theta (-v,.)$ 
has a simple zero between any two successive signs $+-$ or $-+$. 

As $v$ increases, for $v^4=\tilde{q}_{\nu}$, the zeros $y_{-2\nu}$ and 
$y_{-2\nu +1}$ of $\psi _1$ and the zeros $z_{-2\nu}$ and 
$z_{-2\nu +1}$ of $\psi _2$ coalesce and these two functions are nonnegative on 
the interval $(y_{-2\nu -1},0)$. Hence 
\vspace{1mm}

1) The two simple zeros of $\theta (-(\tilde{q}_{\nu})^{1/4},.)$, 
which for small values of $v$ belong to $(y_{-2\nu -1},y_{-2\nu +1})$, 
coalesce for 
some $v_0\in (0,(\tilde{q}_{\nu})^{1/4})$, so $\theta (-v_0,.)$ has a double zero 
in the interval $(y_{-2\nu -1},0)$. In fact, in the interval $[z_{-2\nu},0)$, 
because both $\psi _1$ and $\psi _2$ are positive on 
$(y_{-2\nu -1},z_{-2\nu})$. For $v=(v_0)^+$, the double zero of $\theta (-v_0,.)$ 
gives rise to a complex conjugate pair of zeros.  
\vspace{1mm}

2) For some $v_*\in (0,v_0]$, one has $y_{-2\nu}=z_{-2\nu +1}$, so 

$$\psi _1(v_*,y_{-2\nu})\equiv \theta (v_*^4,-y_{-2\nu}^2/v_*)=
\psi _2(v_*,y_{-2\nu})\equiv -v_*y_{-2\nu}\theta (v_*^4,-v_*y_{-2\nu}^2)=0~.$$ 
Hence 
$\theta (-v_*,y_{-2\nu})=0$, and either $y_{-2\nu}$ is a double zero of 
$\theta (-v_*,.)$ (hence $v_*=v_0$) or $\theta (-v_*,.)$ has another 
negative zero which is $>y_{-2\nu -1}$ and one has $v_*\in (0,v_0)$. 
\vspace{1mm}

One can introduce 
the new variable $X:=-vx^2$ and denote by $\cdots <X_{j+1}<X_j<\cdots <0$ 
the zeros of the function $\theta (v^4,X)$. 
We apply to this function Proposition~\ref{propfirsttwo}, 
for $v\in ((\tilde{q}_{\nu})^{1/4},(\tilde{q}_{\nu +1})^{1/4}]$, $\nu \geq 4$.
This gives $X_{2\nu}>-156$ (hence $|X_{2\nu}|<156$). 
Indeed, $X_{2\nu}$ and $X_{2\nu -1}$ are 
the two rightmost of the real negative zeros of $\theta (v^4,X)$. 

On the other hand, one has $\psi _2(v,x)=-vx\theta (v^4,-vx^2)$, i.e.  
$z_{-2\nu}=-(|X_{2\nu}|/v)^{1/2}$, and $v^4\in (\tilde{q}_3,1)=(0.630628\ldots ,1)$, 
hence $v^{1/2}>0.94$ and one can write 

$$z_{-2\nu}>-|X_{2\nu}|^{1/2}/0.94>-\sqrt{156}/0.94=-12.48999600\ldots /0.94>
-13.29~.$$
Thus for $\nu \geq 4$, the two rightmost negative zeros of $\theta (-v,.)$ 
belong to the interval $(-13.29,0)$, and so do the negative double zeros of 
$\theta (-v,.)$ as well whenever $-v$ is a spectral value, 
i.e. $-v=\bar{q}_{\nu}$.
The approximative values of the first three double negative 
zeros of $\theta (\bar{q}_{\nu},.)$   
are $-2.991$, $-3.621$ and $-3.908$, see~\cite{KoPRSE2}. They correspond to 
$\nu =1$, $3$ and~$5$.

Consider now the positive zeros of $\theta$. The analog of inequalities 
(\ref{eqfirststring}) reads:

\begin{equation}\label{eqsecondstring}
\begin{array}{ccccccccc}
&&&0&<&y_{2\nu -1}&<&z_{2\nu -1}&\\ 
&&&&&v^{-4\nu +5/2}&-&v^{-4\nu +3/2}&\\ \\ 
<&y_{2\nu}&<&z_{2\nu}&<&y_{2\nu +1}&<&z_{2\nu +1}&.\\ 
&v^{-4\nu +1/2}&+&v^{-4\nu -1/2}&&v^{-4\nu -3/2}&-&
v^{-4\nu -5/2}&\end{array}
\end{equation}
The two leftmost positive zeros of $\theta (-v,.)$ (and, in particular, 
the double zeros of $\theta (-v,.)$ for $-v=\bar{q}_{\nu}$, $\nu \geq 4$) 
belong to the interval 
$(0,y_{2\nu +1})$ (because $\theta (-v,.)$ has a simple zero between any two 
successive signs $+-$ or $-+$, see the second lines of (\ref{eqsecondstring})). 
Thus one has to find a majoration for $y_{2\nu +1}$. From part (2) of 
Remarks~\ref{remsdivers1} one deduces the inequalities:

$$v^{-8\nu +4}<|X_{2\nu}|<v^{-8\nu}~~~\, \, {\rm and}~~~\, \, 
v^{-8\nu -4}<|X_{2\nu +1}|<v^{-8\nu -8}~.$$
As $z_{2\nu}=(|X_{2\nu}|/v)^{1/2}$ and $z_{2\nu +1}=(|X_{2\nu +1}|/v)^{1/2}$, one obtains 

$$v^{-4\nu +3/2}<z_{2\nu}<v^{-4\nu -1/2}~~~\, \, {\rm and}~~~\, \,  
v^{-4\nu -5/2}<z_{2\nu +1}<v^{-4\nu -9/2}~.$$
Hence $y_{2\nu +1}=vz_{2\nu +1}<v^{-4\nu -7/2}<v^{-5}z_{2\nu}<
(0.630628)^{-5/4}\times 13.29=23.64\ldots$. The 
first three double positive zeros of $\theta (-v,.)$ equal 
$2.907$, $3.523$ and $3.823$, see~\cite{KoPRSE2}. They correspond to 
$-v=\bar{q}_{\nu}$ for 
$\nu =2$, $4$ and~$6$. 
\end{proof}

\section{Behaviour of the complex conjugate pairs
\protect\label{seccomplex}}

We consider first the case $q\in (-1,0)$ in which the results admit 
shorter formulations and proofs. 

\subsection{The case $q\in (-1,0)$\protect\label{subseccomplexqneg}}

We recall first a result which is proved in \cite{KoPMD}:

\begin{tm}\label{tmcplxneg}
For any $q\in (-1,0)$, all zeros of $\theta (q,.)$ belong to the strip 
$\{ |${\rm Im}$x|<132\}$.
\end{tm}
In the present subsection we prove the following result:

\begin{tm}\label{tmcomplexqneg}
For any $q\in (-1,0)$ and for any $y\in \mathbb{R}$, one has 
{\rm Re}$(\theta (q,iy))\neq 0$. Hence the zeros of $\theta$ do not cross 
the imaginary axis.
\end{tm}
It would be interesting to know whether there exists a vertical strip, 
containing in its interior the imaginary axis, in which, 
for any $q\in (-1,0)$, 
$\theta (q,.)$ has no zeros; and whether there exists a compact set 
(consisting 
of two components, one in the left and one in the right half-plane) 
to which belong all complex conjugate pairs of zeros, 
for all $q\in (-1,0)$.  

\begin{proof}
To consider the restriction of $\theta$ to the imaginary axis we 
set $\theta ^{\ddagger}(q,y):=\theta (q,iy)$, $y\in \mathbb{R}$. Clearly, 

\begin{equation}\label{eqdecomp1}
\begin{array}{ccccc}
\theta ^{\ddagger}(q,y)&=&\sum _{j=0}^{\infty}(-1)^jq^{j(2j+1)}y^{2j}&+&
iqy\sum _{j=0}^{\infty}(-1)^jq^{j(2j+3)}y^{2j}\\ \\ 
&=&\theta (q^4,-y^2/q)&+&iqy\theta (q^4,-qy^2)~.\end{array}
\end{equation}
Suppose now that $q\in (-1,0)$. 
To interpret equalities (\ref{eqdecomp1}) easier 
we set $\rho :=-q$ (hence $\rho \in (0,1)$). Thus 

\begin{equation}\label{eqdecomp2}
\theta ^{\ddagger}(q,y)=\theta (q^4,-y^2/q)+iqy\theta (q^4,-qy^2)=
\theta (\rho ^4,y^2/\rho )-
i\rho y\theta (\rho ^4,\rho y^2)~.
\end{equation}
Both the real and the imaginary parts of $\theta ^{\ddagger}$ 
are expressed as values of 
$\theta (q,x)$ for $q\in (0,1)$, $x\geq 0$ (with $q=\rho ^4$ and 
$x=y^2/\rho$ or $x=\rho y^2$). Hence the real part of $\theta ^{\ddagger}$ 
is nonzero for all 
$y\in \mathbb{R}$ (because $\theta (q,x)>0$ for $q\in (0,1)$, $x\geq 0$) 
which means 
that for $q\in (-1,0)$, the zeros of $\theta$ do not 
cross the imaginary axis.
\end{proof}

\subsection{The case $q\in (0,1)$\protect\label{subseccomplexqpos}}

We remind first a result from \cite{KoPMD}:

\begin{tm}\label{tmcplxpos}
For any value of the 
parameter $q\in (0,1)$, all zeros of the function $\theta (q,.)$ belong to 
the domain 
$\{ {\rm Re}~x<0$, $|{\rm Im}~x|<132\}~\cup ~\{ {\rm Re}~x\geq 0, 
|x|<18\}$.
\end{tm}

In this subsection we prove the following theorem;

\begin{tm}\label{tmcross}
There are infinitely-many values of $q\in (0,1)$ (tending to $1$) for which 
a complex conjugate pair of zeros of $\theta (q,.)$ crosses the imaginary axis 
from left to right. Not more than finitely-many complex conjugate pairs 
of zeros cross the imaginary axis from right to left. 
\end{tm}

\begin{conj}\label{conjcross}
For every $j\in \mathbb{N}$, the complex conjugate pair born for 
$q=(\tilde{q}_j)^+$ crosses for some $q^*_j\in (\tilde{q}_j,1)$ 
the imaginary axis from left to right. No complex conjugate 
pair crosses the imaginary 
axis from right to left.
\end{conj}

\begin{rems}
{\rm (1) Conjecture~\ref{conjcross} (if proved) 
combined with Theorem~\ref{tmcplxpos} 
would imply that all complex conjugate pairs, after having crossed 
the imaginary axis, remain in the half-disk $\{$~Re$~x\geq 0, 
|x|<18~\}$. Theorem~\ref{tmcross} allows to claim this about 
infinitely-many of these pairs.

(2) Recall that $1-\tilde{q}_j\sim \pi /2j$, see part (3) of 
Remarks~\ref{remsdivers1}. 
Denote by $C(q)$ the quantity of all complex conjugate pairs 
of $\theta$ for $q\in (0,1)$ and by $CR(q)$ the quantity of such pairs 
with nonnegative real part. Hence one can expect that 
$\lim _{q\rightarrow 1^-}(CR(q)/C(q))=1/4$. This can be deduced from the proof 
of Theorem~\ref{tmcross} below in which we use formula (\ref{eqdecomp2}). 
The values of the argument $v^4$ corresponding to the moments when a complex 
conjugate pair crosses the imaginary axis are expected to be 
of the form $1-\pi /2j+o(1/j)$ hence $v=1-\pi /8j+o(1/j)$. Thus asymptotically, 
as $j\rightarrow \infty$, for $q\in (1-\pi /2j,1-\pi /2(j+1)]$, 
one should have $C\sim j$ and $CR\sim j/4$. That is, 
crossing of the imaginary axis 
by a complex conjugate pair should occur four times 
less often than birth of such a pair.} 
\end{rems}

\begin{proof}[Preparation of the proof of Theorem~\ref{tmcross}]
We precede the proof of Theorem~\ref{tmcross} by the present observations 
which are crucial for the understanding of the proof. We shall be using 
equalities (\ref{eqdecomp1}), but with $q\in (0,1)$. The condition 
$\theta (q,iy)=0$ for some $y\in \mathbb{R}$ indicates  
the presence of a zero of $\theta$ on the imaginary axis. One can introduce 
the new variables $q_{\circ}:=q^4$ and $Y:=-y^2/q$. Hence, supposing that $y>0$, 
the right-hand side of 
(\ref{eqdecomp1}) is of the form 

$$\theta (q_{\circ},Y)+
i(q_{\circ})^{1/4}(-(q_{\circ})^{1/4}Y)^{1/2}
\theta (q_{\circ},(q_{\circ})^{1/2}Y)~.$$
Thus (writing $(q,x)$ instead of $(q_{\circ},Y)$) 
we have to consider the zero sets of the functions 
$\theta (q,x)$ and $\theta (q,\sqrt{q}x)$. These sets are shown, in solid 
and dashed lines respectively, on Fig.~\ref{twographs}. 

\begin{figure}[htbp]
\centerline{\hbox{\includegraphics[scale=0.7]{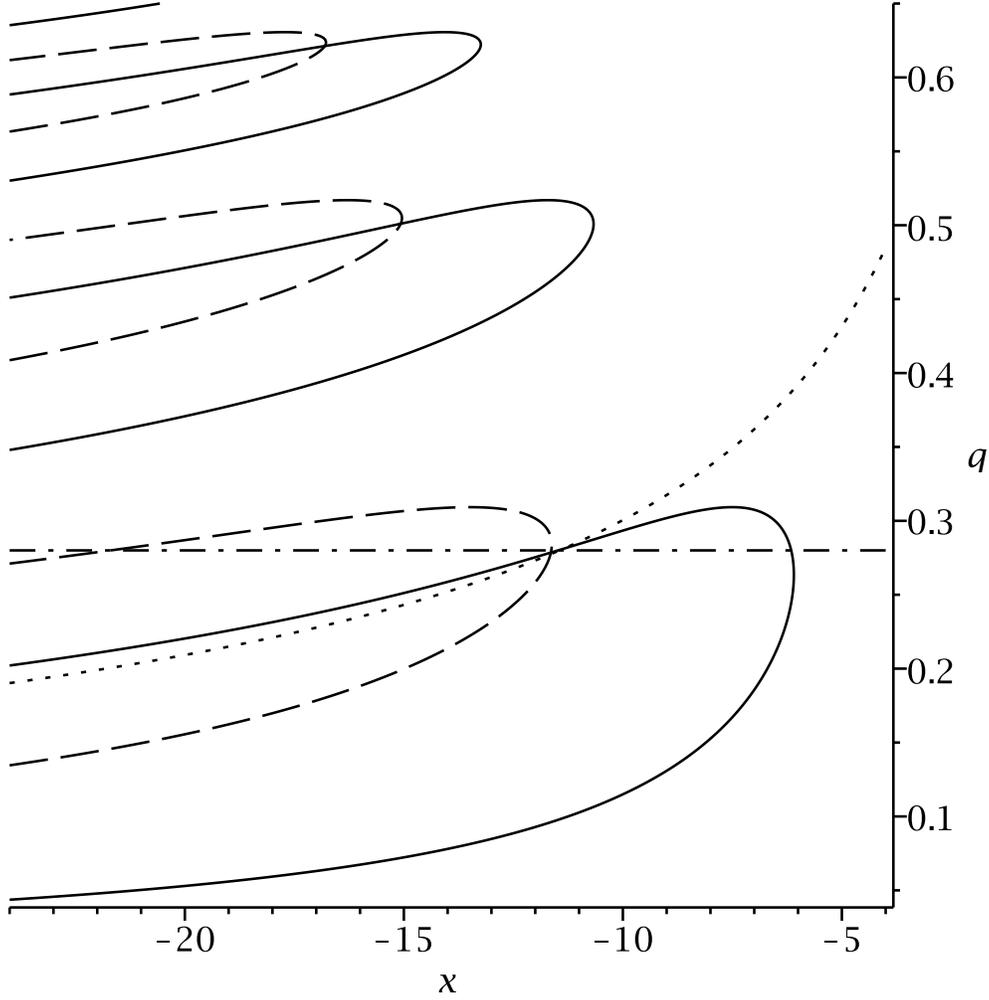}}}
    \caption{The real-zeros sets of $\theta (q,x)$ and $\theta (q,\sqrt{q}x)$.}
\label{twographs}
\end{figure}
Recall that the curves $\Gamma _j$ were defined in Theorem~\ref{tmgeom}. 
One can define by analogy the curves $\Gamma _j^*$ for the set 
$\{ (q,x)~|~\theta (q,\sqrt{q}x)=0\}$. On Fig.~\ref{twographs} one can see the 
intersection points of $\Gamma _j$ and $\Gamma _j^*$ for $j=1$, $2$ and $3$. 
Through the point $Z_1:=\Gamma _1\cap \Gamma _1^*$ passes 
a curve $x=-q^{-\gamma}$ with $\gamma \in (1,2)$. On Fig.~\ref{tmgeom} we 
represent this curve by dotted line and we draw by dash-dotted line 
the horizontal line $\mathcal{L}_1:\{ q=q_{\bullet}\}$ 
passing through the point $Z_1$. The line $\mathcal{L}_1$ intersects each of 
the curves $\Gamma _1$ and $\Gamma _1^*$ at two points one of which is $Z_1$ 
(the left of the two points for $\Gamma _1$ and the right of the two for 
$\Gamma _1^*$). If one considers the graphs of the functions 
(in the variable~$x$)
$\theta (q_{\bullet},x)$ and 
$\theta (q_{\bullet},\sqrt{q_{\bullet}}x)$, 
then they will look like the two graphs drawn in solid line above left on 
Fig.~\ref{PTconjpairs}; the point $Z_1$ will be the point $B$ on 
Fig.~\ref{PTconjpairs}. Nevertheless one should keep in mind that 
Fig.~\ref{PTconjpairs} represents the graphs of two functions whose arguments 
are of the form $-qy^2$ and $-y^2/q$, i.e. increasing of $y>0$ corresponds to 
the decreasing of the (negative) values of these arguments.  

The curve $x=-q^{-\gamma}$ and the line 
$\mathcal{L}_1$ were defined in relationship with $\Gamma _1$ and $\Gamma _1^*$, 
i.e. for $j=1$. One can consider their analogs defined for 
$j=2$, $3$, $\ldots$. Recall that the 
quantities $\kappa ^{\triangle}$ and $q^{\triangle}$ were defined 
in Remarks~\ref{remsKdagger}. It is only for $j$ sufficiently large 
($j\geq (\kappa ^{\triangle}+1)/2$) that we have proved that the intersection of 
the curve $\Gamma _j$ with each curve $x=-q^{-a}$, $a\geq \kappa ^{\triangle}$, 
is a point or is empty (see Proposition~\ref{propKdagger}). For smaller values 
of $j$ 
we can claim only that this intersection (of two analytic curves) 
consists of not more than a finite 
number of points. This explains the final sentence of Theorem~\ref{tmcross}.
\end{proof}

\begin{proof}[Proof of Theorem~\ref{tmcross}]
To study the restriction of $\theta$ to the imaginary axis we 
set again 

$$\theta ^{\ddagger}(q,y):=\theta (q,iy)=f_1(q,y)+iqyf_2(q,y)~~~\, ,~~~\,  
y\in \mathbb{R}~,$$ 
where $f_1(q,y):=\theta (q^4,-y^2/q)$ and $f_2(q,y):=\theta (q^4,-qy^2)$, 
see equalities (\ref{eqdecomp1}) (in which we assume that $q\in (0,1)$).   
For $q$ close to $0$, the zeros 
of $\theta (q,x)$ are close to the numbers $-q^{-j}$, $j\in \mathbb{N}$.    
More precisely, 
for $q\in (0,0.108]$, there is a simple zero of $\theta$ of the form 
$-1/(q^j\Delta _j)$ with $\Delta _j\in [0.2118,1.7882]$, and all 
these zeros are distinct, see Theorem~2.1 in~\cite{KoPRSE1}. 
Thus the zeros of $f_1(q,y)$ 
(resp. of $f_2(q,y)$) 
are close to the quantities $\pm q^{-2j+1/2}$ (resp. $\pm q^{-2j-1/2}$); 
hence the positive (resp. negative) zeros of $f_1$ interlace 
with the positive (resp. negative) zeros of $f_2$.

For $q>0$ small enough, all zeros of $\theta (q,.)$ are real negative; 
hence all zeros of $f_1$ are real. For such values of $q$, we denote 
the positive zeros of 
$f_1$ by $y^{\sharp}_j$, $y^{\sharp}_j<y^{\sharp}_{j+1}$ 
($y^{\sharp}_j$ is close to $q^{-2j+1/2}$). As $q$ increases, 
these zeros depend continuously on $q$ and as we will see below, 
certain couples of them, for some values of $q$, 
coalesce and form complex conjugate pairs. Thus their indices are meaningful 
only till the value of $q$ corresponding to the moment of confluence.

As 
$f_1(q,y)=f_2(q,y/q)$, the positive zeros of $f_2$ equal 
$y^{\sharp}_j/q$. 
Consider the zeros $y^{\sharp}_{2j-1}$ and $y^{\sharp}_{2j}$ of $f_1$ and the zeros 
$y^{\sharp}_{2j-1}/q$ and $y^{\sharp}_{2j}/q$ of $f_2$. 
For values of $q$ close to $0$, 
they satisfy the following inequalities 
(we indicate in the second row the powers of $q$ 
to which they are approximatively equal for $q$ close to $0$):

\begin{equation}\label{ineqy1}
\begin{array}{cccccccc}
y^{\sharp}_{2j-1}&<&y^{\sharp}_{2j-1}/q&<&y^{\sharp}_{2j}&<&y^{\sharp}_{2j}/q&.\\ \\ 
q^{-4j+5/2}&&q^{-4j+3/2}&&q^{-4j+1/2}&&q^{-4j-1/2}&\end{array}
\end{equation}
As $q$ increases, for $q^4=\tilde{q}_j$ (i.e. for $q=(\tilde{q}_j)^{1/4}$), 
the zeros $y^{\sharp}_{2j-1}$ and $y^{\sharp}_{2j}$ of $f_1$ (hence the zeros 
$y^{\sharp}_{2j-1}/q$ and $y^{\sharp}_{2j}/q$ of $f_2$ as well) 
coalesce and then give 
birth to a complex conjugate pair. This means that for values of $q$ 
just before the moment of confluence one has 

\begin{equation}\label{ineqy2}
y^{\sharp}_{2j-1}<y^{\sharp}_{2j}<y^{\sharp}_{2j-1}/q<y^{\sharp}_{2j}/q~.
\end{equation}
Therefore there exists $q^{\dagger}_j\in (0,(\tilde{q}_j)^{1/4})$ for which one has 
$y^{\sharp}_{2j-1}<y^{\sharp}_{2j}=y^{\sharp}_{2j-1}/q<y^{\sharp}_{2j}/q$, 
i.e. the real and imaginary parts 
of $\theta$ have a common zero $y^{\sharp}_{2j}=y^{\sharp}_{2j-1}/q$. 
This is a simple 
zero both for $f_1$ and $f_2$ 
(hence $iy$ is a simple zero of $\theta (q^{\dagger}_j,.)$). 
Indeed, $y^{\sharp}_{2j}$ 
can be either a simple or 
a double zero of $f_1$; if it is a double one, then  
$y^{\sharp}_{2j}$ must coalesce with $y^{\sharp}_{2j-1}$ 
(this follows from part (2) of Theorem~\ref{tmknown1} and part (1) of 
Remarks~\ref{remsdivers1}); this happens for $q^4=\tilde{q}_j$ 
which contradicts $q^{\dagger}_j<(\tilde{q}_j)^{1/4}$.  
  
We have just shown that for $j\in \mathbb{N}$ (i.e. for infinitely-many 
values 
of $q=q^{\dagger}_j\in (0,1)$) the function $\theta$ has a simple 
conjugate pair of zeros 
on the imaginary axis. 
In what follows we can assume that $j\geq (\kappa ^{\triangle}+1)/2$, see 
Remarks~\ref{remsKdagger} 
and the above preparation of the proof of Theorem~\ref{tmcross}. 
Thus the quantity $q^{\dagger}_j$ is unique.
The real zeros $\xi _j$  
of $\theta (q,x)$ for $q\in (0,1)$ 
satisfy the following string of inequalities: 

$$-q^{-2j}<\xi _{2j}<\xi _{2j-1}<-q^{-2j+1}$$
(see equation (6) in \cite{KoBSM1}). This implies the inequalities 

$$-q^{-8j}<\tau _{2j}<\tau _{2j-1}<-q^{-8j+4}~,$$
satisfied by the zeros of $\theta (q^4,x)$, and as 
$y^{\sharp}_j=(-q\tau _j)^{1/2}$, the inequalities 

\begin{equation}\label{ineqy3}
\begin{array}{lcl}
q^{-4j+5/2}<y^{\sharp}_{2j-1}<y^{\sharp}_{2j}<q^{-4j+1/2}&&{\rm and}\\ \\ 
q^{-4j+3/2}<y^{\sharp}_{2j-1}/q<y^{\sharp}_{2j}/q<q^{-4j-1/2}&& 
\end{array}
\end{equation} 
hold true. Inequalities (\ref{ineqy3}) (see also (\ref{ineqy1}) and 
(\ref{ineqy2})) imply that 
\vspace{1mm}

{\em i)} the zero 
$y^{\sharp}_{2j}$ of $f_1$ can be equal to $y^{\sharp}_{2j-1}/q$ 
and to no other zero of $f_2$ and 

{\em ii)} the zero $y^{\sharp}_{2j-1}$ 
of $f_1$ can be equal to neither of the zeros of $f_2$. 
\vspace{1mm}

On Fig.~\ref{PTconjpairs} (above left, in solid line) 
we show the graphs of the functions 
$f_1(q^{\dagger}_j,.)$ and $f_2(q^{\dagger}_j,.)$ 
(they are denoted by ``Re'' and ``Im'' respectively). 
The points $A$, $B$ and $C$ indicate the positions of the zeros 
$y^{\sharp}_{2j-1}$, $y^{\sharp}_{2j}=y^{\sharp}_{2j-1}/q$ and $y^{\sharp}_{2j}/q$ 
respectively. By dotted lines we show these graphs for 
$q\in (q^{\dagger}_j,(\tilde{q}_j)^{1/4})$. 
We remind that (see the proof of Theorem~1 in \cite{KoBSM1}) as $q$ increases, 
the local minima of $\theta (q,.)$ 
go up; when $q$ runs over an interval $((\tilde{q}_j)^-,(\tilde{q}_j)^+)$, 
the two rightmost real zeros coalesce for $q=\tilde{q}_j$ and the function 
$\theta (\tilde{q}_j,.)$ has a local minimum at this double zero.

 \begin{figure}[htbp]
\centerline{\hbox{\includegraphics[scale=0.7]{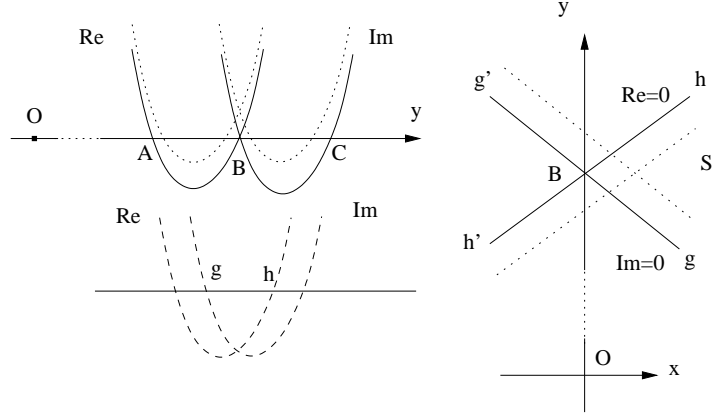}}}
    \caption{The complex conjugate pairs of $\theta$ on the imaginary axis.}
\label{PTconjpairs}
\end{figure}
Consider the function $\theta (q,iy+\varepsilon )$, i.e. the restriction of 
$\theta$ to a line in the $x$-plane parallel to the imaginary axis 
and belonging to the right half-plane. One checks directly that 

$$\begin{array}{cclc}
\theta (q,iy+\varepsilon )&=&\theta (q,iy)+
\varepsilon (\partial \theta /\partial x)(q,iy)
+o(\varepsilon )&{\rm where}\\ \\ 
(\partial \theta /\partial x)(q,iy)&=&K_1+iK_2&,\\ \\ 
K_1&:=&\sum _{\nu =0}^{\infty}(-1)^{\nu}(2\nu +1)
q^{(\nu +1)(2\nu +1)}y^{2\nu}&\\ \\ 
&=&2(-q^2y^2)(\partial \theta /\partial x)
(q^4,-qy^2)+q\theta (q^4,-qy^2)&,\\ \\ 
K_2&:=&\sum _{\nu =1}^{\infty}(-1)^{\nu}2\nu 
q^{\nu (2\nu +1)}y^{2\nu -1}&\\ \\ 
&=&2(y/q)(\partial \theta /\partial x)(q^4,-y^2/q)&.
\end{array}$$
The second term of $K_1$ vanishes at $B$. The 
first term equals 

$$qy(-2qy)(\partial \theta /\partial x)(q^4,-qy^2)=
qy(\partial f_2/\partial y)~.$$
As Im\,$\theta$ is decreasing at $B$, one sees that $K_1<0$. In the 
same way, 

$$K_2=-(-2y/q)(\partial \theta /\partial x)(q^4,-y^2/q)=
-\partial f_1/\partial y~.$$
Looking at the graph of 
Re\,$\theta$ at the point $B$ 
one sees that Re\,$\theta$ is increasing there and one 
concludes that $K_2<0$. 

On the right-hand of Fig.~\ref{PTconjpairs}, we 
represent in solid line the sets Re\,$\theta =0$ and 
Im\,$\theta =0$ (in the $x$-plane, 
close to the point $B$ of the imaginary axis). 
These are the segments $hh'$ and $gg'$ respectively. 
The true sets are in fact not straight lines, but 
arcs whose tangent lines at $B$ 
look like $hh'$ and $gg'$; as $q$ varies, these arcs and their tangent lines 
change continuously. 

Thus for $q=q^{\dagger}_j$, 
both quantities $K_1$ and $K_2$ are negative in the sector $S$. 
The graphs of Re\,$\theta (q,iy+\varepsilon )$ and 
Im\,$\theta (q,iy+\varepsilon )$ (considered as functions in $y$, 
for fixed $q$ and $\varepsilon$) are represented by dashed lines 
to the left below on Fig.~\ref{PTconjpairs}. 

As $q$ increases on $[q^{\dagger}_j,(q^{\dagger}_j)^+)$, 
the values of Re\,$\theta$ and Im\,$\theta$ along the arcs 
$hh'$ and $gg'$ become positive (this corresponds to the fact that close 
to the point $B$, 
the graphs of $f_1$ and $f_2$ are above the $y$-axis for $q=(q^{\dagger}_j)^+$, 
see the dotted graphs above left on Fig.~\ref{PTconjpairs}). Hence the 
sets  Re\,$\theta =0$ and 
Im\,$\theta =0$ shift as shown by dotted line on the right of 
Fig.~\ref{PTconjpairs} (it would be more exact to say that the tangent 
lines to these sets at $B$ shift like this). That is, their intersection point 
is in the right half-plane and the complex conjugate pair 
crosses the imaginary axis from left to right.

\end{proof}

\end{document}